\documentclass[reqno,11 pt]{amsart}
\usepackage{amsmath,amssymb,latexsym}
\usepackage[colorlinks=true, linkcolor={black}, citecolor=purple]{hyperref}
\usepackage{amsthm}
\usepackage{comment}
\usepackage{color}  
\usepackage[all]{xy}
\usepackage{tikz-cd}
\usepackage{multirow}
\usepackage{multicol}
\usepackage{longtable}
\usepackage{caption}
\usepackage{float}
\usetikzlibrary{cd}
\usepackage[
    style=alphabetic,
    sorting=nyt,
    maxbibnames=99,
    giveninits=false,
    url=false,
    doi=false,
    isbn=false
]{biblatex}
\hypersetup{
    colorlinks=true, %set true if you want colored links
    linktoc=all,     %set to all if you want both sections and subsections linked
    linkcolor=blue,  %choose some color if you want links to stand out
    urlcolor=black,
}

\newcommand{\F}{{\mathbb{F}}}
\newcommand{\Z}{{\mathbb{Z}}}
\newcommand{\Q}{{\mathbb{Q}}}

\newcommand{\Sel}{\mathrm{Sel}}

 \newcommand{\genlegendre}[4]{%
  \genfrac{(}{)}{}{#1}{#3}{#4}%
  \if\relax\detokenize{#2}\relax\else_{\!#2}\fi
}
\newcommand{\legendre}[3][]{\genlegendre{}{#1}{#2}{#3}}

\makeatletter
\newcommand*{\rom}[1]{\expandafter\@slowromancap\romannumeral #1@}
\makeatother

\DeclareFontFamily{U}{wncy}{}
    \DeclareFontShape{U}{wncy}{m}{n}{<->wncyr10}{}
    \DeclareSymbolFont{mcy}{U}{wncy}{m}{n}
    \DeclareMathSymbol{\Sh}{\mathord}{mcy}{"58}

\theoremstyle{plain}
\newtheorem{theorem}{Theorem}[section]
\newtheorem{lemma}[theorem]{Lemma}

\newtheorem*{theorem*}{Theorem}
\newtheorem*{ass*}{Assumption}

\numberwithin{equation}{section}

\begin{document}
\title[Monsky Matrix and 2-Selmer rank]{Monsky Matrix and 2-Selmer rank}

\author[Shamik Das]{Shamik Das}
\email{shamikd@iitk.ac.in}

\author[Sudipa Mondal]{Sudipa Mondal}
\email{sudipa.mondal123@gmail.com}

%\author[Somnath Jha]{Somnath Jha}
%\email{jhasom@iitk.ac.in}

\begin{abstract}
In this article, we produce infinite families of non-congruent numbers in the residue class of $1,2,$ and $3$ modulo $8$ with arbitrarily many triples or quadruples prime factors. In short, we use Monsky matrix to show that the $2$-Selmer rank of the corresponding congruent number elliptic curve is zero. We also establish some quantitative results to conclude that each such family contains infinitely many non-congruent numbers.

%Monsky matrix and non-congruent number. Li and Qin \cite{Li-Qin}
 \end{abstract}
	

\maketitle

\section{Introduction}
A rational number $n$ is called a congruent number if it appears as the area of a right-angled triangle. More precisely, $n$ is said to be a congruent number if there exist rational numbers $a,b,c$ such that $a^2+b^2=c^2$ and $n=\frac{1}{2}ab$. It is easy to see that $n$ is a congruent number if and only if $nd^2$ is a congruent number, where $d$ is any number. Hence we only focus on positive square-free integers. It turns out that the congruent number property is also equivalent to existence of a non-trivial solution (of infinite order) of the elliptic curve $E_n: y^2=x^3-n^2x$ over $\Q$. It is well-known that there is a one-to-one correspondence between the following two sets:
 \begin{equation*}
    \left\{(a,b,c): a^2+b^2=c^2,\; n=\frac{1}{2}ab \right\}, \, \, \,\text{ and }~~~ \{(x,y): y^2=x^3-n^2x,\; y\neq 0\}.
\end{equation*}
The elliptic curve $E_n$ is known as the congruent number elliptic curve. By Mordell-Weil theorem, we know that the set of rational points $E_n(\Q)$ is finitely generated abelian group, and we denote the Modell-Weil rank of $E_n(\mathbb{Q})$ by $r(n)$. The bridge between the elliptic curves and the congruent number problem is the following:
$$n \text{ is a congruent number} \iff r(n) \geq 1.$$
In short, a positive square-free integer $n$ is called a non-congruent number if and only if Mordell-Weil rank $r(n)$ of the corresponding congruent number elliptic curve $E_{n}(\mathbb{Q})$ is zero.

Let $L(E_n,s)$ be the Hasse-Weil $L$-function in the complex variable $s$ attached to $E_n$ which is conjectured to be analytically continued in the whole complex plane. Let the order of vanishing of $L(E_n,s)$ at $s=1$ is denoted by $R(n)$, which is called the analytic rank of $E_n(\Q)$. Then the first part of the Birch and Swinnerton-Dyer (BSD) conjecture predicts that $r(n)=R(n)$. Further calculations by Birch and Stephen \cite{bs} showed that \begin{equation*}
R(n) \equiv
\begin{cases}
      0 \pmod 2 ~~\text{ if } & n \equiv 1,2,3 \pmod 8, \\
     1 \pmod 2 ~~ \text{ if } & n \equiv 5,6,7 \pmod 8.  
\end{cases}  
\end{equation*}
Therefore assuming BSD conjecture, one can predict that the integers $n \equiv 5,6,7 \pmod 8$ are congruent numbers. Hence the cases $n \equiv 1,2,3 \pmod8$ become more interesting. Although there is no algorithm to determine a number to be congruent or not, many authors studied the congruent number problem and obtained significant results. In 1855, Genocchi \cite{genocchi1855note} considered primes of the form $8k+3$ and proved that they are non-congruent numbers. Furthermore, in the same article, Genochhi proved that the numbers $2p$ where the primes $p$ are of the form $8k+5$ are non-congruent numbers.  In 1975, Lagrange \cite{lagrange} considered certain family of numbers which are product of either two primes or three primes and provided a criterion for them to be non-congruent numbers. In \cite{dm-joaa}, \cite{liqin}, \cite{zhang}, the authors provide  criteria to construct families of non-congruent numbers involving the $2$-part of the class number of certain imaginary quadratic fields. Their proofs heavily rely on the $2$-descent of elliptic curves. In another direction, the authors in \cite{weidong}, \cite{ds},  \cite{Das-Saikia-Rocky}, \cite{ds2}, \cite{junguklee},  \cite{Li-Qin}, \cite{lindsey2}, \cite{lindsey} present a method that generates new families of non-congruent numbers from existing families of non-congruent numbers. They show that infinitely many new non-congruent numbers can be constructed by appending primes of a certain form to existing non-congruent numbers whose corresponding congruent elliptic curves have the 2-Selmer rank equal to zero. The results can be used to generate even non-congruent numbers with prime factors in certain odd congruence class modulo eight. The authors rely on Monsky's formula to calculate the 2-Selmer rank of the congruent elliptic curve.

%In \cite{Li-Qin}, \cite{ds}, \cite{ds2}, \cite{Das-Saikia-Rocky} the authors found a series of family of non-congruent numbers which are product of arbitrarily many primes. The authors in \cite{zhang}, \cite{dm-joaa} provided some criterion to construct non-congruent numbers. Recent works by \cite{lindsey}, \cite{junguklee}, \cite{weidong}, \cite{lindsey2} are also about finding families of non-congruent numbers.

In this article, we produce families of non-congruent numbers with arbitrarily many triples or quadruples prime factors along with some relation involving the Legendre symbols of the primes by showing that the $2$-Selmer rank of the corresponding congruent number elliptic curve is zero. Our results have been stated in \S\ref{sott}.

\textbf{Methodology and structure of the article:} It is well-known that the Modell-Weil rank $r(n)$ is bounded by the $2$-Selmer rank (see \S \ref{preli} for details). To prove a square-free positive integer to be a non-congruent number, it is enough to show that the $2$-Selmer rank for the corresponding elliptic curve $E_n$ is zero. We relate the $2$-Selmer rank of $E_n$ with the rank of the Monsky matrix as in \cite[Appendix]{hb} and prove that the corresponding Monsky matrix has full rank to construct families of non-congruent numbers. Finally we show that density of such families to be positive which asserts that there are infinitely many non-congruent numbers in these families. 

The sturcture of this paper is as follows: In \S \ref{sott}, we present a sequence of theorems that yields families of non-congruent numbers. In \S \ref{preli}, we study the Modell-Weil rank and $2$-Selmer rank for the congruent number elliptic curve, and recall the relation between the $2$-Selmer rank and rank of the associated Monsky matrix. We gather several basic results in  \S  \ref{some pre lemmas} and using these results, we prove the main theorem in \S \ref{proof of the theorems}. Finally, \S \ref{infinite families and examples} presents asymptotic results on the number of non-congruent numbers in these families, along with explicit examples (see Table \ref{tab:examples}) of such non-congruent numbers. \\

\noindent{\it Acknowledgement}: The authors would like to thank Prof. Somnath Jha for his suggestions and comments. The first author acknowledges the support of the DST-INSPIRE Faculty Fellowship\\ DST/INSPIRE/04/2024/004189 at IIT Kanpur. The second author acknowledges the support provided by the Institute Postdoctoral fellowship at IIT Madras. % The authors also thank Professor Somnath Jha for useful discussion.

\section{Statement of main results}\label{sott}
This section is devoted to stating our main results. Let $\legendre{k}{l}$ denote the Legendre symbol, where $l$ is an odd prime and $k$ is any integer. Suppose that $p, q, r$ and $s$ are primes. Our results may be viewed as generalizations of the following classical results of Lagrange \cite{lagrange}:

\begin{itemize}
    \item[(1)] If $(p, q, r) \equiv (1, 5, 7) \pmod{8}$ (respectively, $(p, q, r) \equiv (3, 5, 5) \pmod{8}$) and 
    $\legendre{p}{q} = \legendre{p}{r} = -1$, then $pqr$ is not a congruent number. 
    These are generalized in Theorem~\ref{157} (respectively, Theorem~\ref{355}).  

    Moreover, if $(p, q, r) \equiv (5, 3, 3) \pmod{8}$ with 
    $\legendre{p}{q} = \legendre{p}{r} = -1$, then $2pqr$ is not a congruent number; 
    this is generalized in Theorem~\ref{533}.

    \item[(2)] If $(p, q, r) \equiv (7, 7, 3) \pmod{8}$ and 
    $\legendre{p}{q} = \legendre{q}{r} = \legendre{r}{p}$, then $pqr$ is not a congruent number. 
    This result is generalized in Theorem~\ref{377}.

    \item[(3)] If $(p, q) \equiv (1, 3)$ or $(5, 7) \pmod{8}$ with 
    $\legendre{p}{q} = -1$, then $pq$ is not a congruent number; 
    see Theorem~\ref{1357} for the generalization.  

    Similarly, if $(p, q) \equiv (1, 5)$ or $(3, 7) \pmod{8}$ with 
    $\legendre{p}{q} = -1$, then $2pq$ is not a congruent number; 
    this is generalized in Theorem~\ref{2times1357}.
\end{itemize}

We prove the following results:

\begin{theorem}\label{157}
Let $p_1,p_2,\ldots ,p_t$; $q_1, q_2, \ldots ,q_t$; and $r_1, r_2, \ldots, r_t$ be distinct primes such that $(p_i, q_i, r_i)\equiv (1,5,7) \pmod 8$ for $1 \leq i \leq t$. Assume that $\alpha \in \{ 1, -1\}$ and 
\begin{enumerate}
	\item[(a)] $\legendre{p_i}{q_i}= \legendre{p_i}{r_i}= -1$ for all $i$,
	\item[(b)] $\legendre{p_i}{p_j}=\legendre{q_i}{q_j}= 1$ and  $\legendre{r_i}{r_j}=\alpha$ for $i < j$,
	\item[(c)] $\legendre{p_i}{q_j} = \legendre{p_i}{r_j}=\legendre{q_i}{r_j}=1$ if $i\neq j$.
\end{enumerate}
Then $n= \prod\limits_{i=1}^{t}p_iq_ir_i$ is not a congruent number. 
\end{theorem}

\begin{theorem}\label{355}
    Let $\alpha \in \{-1,1\}$ and $p_1,p_2,\ldots ,p_t$; $q_1, q_2, \ldots ,q_t$; and $r_1, r_2, \ldots, r_t$ be distinct primes such that $(p_i, q_i, r_i)\equiv (3,5,5) \pmod 8$ for $1 \leq i \leq t$.  Assume that 
    \begin{enumerate}
        \item[(a)] $\legendre{p_i}{q_i}= \legendre{p_i}{r_i} = -1 $ for all $i$,
        \item[(b)] $\legendre{q_i}{q_j} =\legendre{r_i}{r_j}=1, \quad \legendre{p_i}{p_j}=\alpha $ for $i<j$,
        \item[(c)] $\legendre{p_i}{q_j} = \legendre{p_i}{r_j}=\legendre{q_i}{r_j}=1$ if $i\neq j$.
    \end{enumerate}
    Then $n= \prod\limits_{i=1}^{t}p_iq_ir_i$ is not a congruent number.
\end{theorem}

\begin{theorem}\label{377}
    Let $t$ be odd and $\alpha, \mu \in \{1,-1\}$. Consider the primes $r_1, r_2, \ldots, r_t$ such that $r_i \equiv 3 \pmod8$ for $1 \leq i \leq t$. Let $p, q \equiv 7 \pmod8$ be another distinct primes satisfying the following :
    \begin{enumerate}
        \item[(a)] $\legendre{p}{q}=\legendre{q}{r_i}=\legendre{r_i}{p}=\alpha$ for all $i$,
        \item[(b)] $\legendre{r_i}{r_j}= \mu$ for $i<j$. 
    \end{enumerate}
    Then $n=pq\prod\limits_{i=1}^{t} r_i$ is not a congruent number.
\end{theorem}

\begin{theorem}\label{533}
Let $p_1,\dots,p_t$; $q_1,\dots,q_t$; and $r_1,\dots,r_t$ be distinct primes with 
$(p_i,q_i,r_i)\equiv (5,3,3)\pmod{8}$ for $1 \leq i \leq t$, and let 
$\alpha,\mu,\mu_1,\mu_2\in \{1,-1\}$. Assume that 
\begin{enumerate}
\item[(a)] $\legendre{p_i}{q_i}=\legendre{p_i}{r_i}=\alpha$, \, $\legendre{q_i}{r_i}=\mu$ for all $i$,
\item[(b)] $\legendre{p_i}{p_j}=1$, \ $\legendre{q_i}{q_j}=\mu_1$, \ $\legendre{r_i}{r_j}=\mu_2$ for all $i<j$,
\item[(c)] $\legendre{p_i}{q_j}=\legendre{p_i}{r_j}=\legendre{q_i}{r_j}=1$ for all $i\ne j$.
\end{enumerate}
Then $n=2\prod\limits_{i=1}^{t} p_iq_ir_i$ is not a congruent number if one of the following holds:
\begin{itemize}
\item[\textbf{Case A:}] $\alpha=1$, and
\begin{itemize}
\item[(i)] $\mu=1$, or
\item[(ii)] $\mu=-1$ with $\mu_1=\mu_2$ and $(t=1 \text{ or } t \text{ even})$, or $\mu_1\ne\mu_2$
.
 
\end{itemize}

\item[\textbf{Case B:}] $\alpha=-1$, and
\begin{itemize}
\item[(i)] $\mu=-1$, or
\item[(ii)] $\mu=1$ with $\mu_1=\mu_2$ and $(t=1 \text{ or } t \text{ even})$, or $\mu_1\ne\mu_2$.
\end{itemize}
\end{itemize}
\end{theorem}

\begin{theorem}\label{1357}
Let $p_1, p_2, \ldots, p_t$; $q_1, q_2, \ldots, q_t$; $r_1, r_2, \ldots, r_t$; and $s_1, s_2, \ldots, s_t$ be distinct primes satisfying $(p_i, q_i, r_i, s_i) \equiv (1,3,5,7) \pmod{8}$ for  $1 \leq i \leq t$, and let $\mu_1, \mu_2 \in \{-1,1\}$. Suppose that
\begin{enumerate}
    \item[(a)] $\legendre{p_i}{q_i}=\legendre{r_i}{s_i}=-1$, and $\legendre{p_i}{r_i}=\legendre{p_i}{s_i}=\legendre{q_i}{r_i}=\legendre{q_i}{s_i}=1$ for all $i$,
    %\item[(b)] $\legendre{q_i}{s_j}=1$ for all $1 \leq i,j \leq t$,
    \item[(b)] $\legendre{a_i}{b_j}=1$ for all $i \neq j$ and $a,b \in \{p,q,r,s\}$ with $a \neq b$ and $(a,b) \neq (s,q)$,
    \item[(c)] $\legendre{s_i}{s_j}=\mu_1$, $\legendre{q_i}{q_j}=\mu_2$ for all $i < j$, and $\legendre{p_i}{p_j}=\legendre{r_i}{r_j}=1$ for all $i \neq j$.
\end{enumerate}
Then $n=\prod\limits_{i=1}^{t} p_iq_ir_is_i$ is not a congruent number.
\end{theorem}

\begin{theorem}\label{2times1357}
    Let $t$ be odd and $p_1,\; p_2,\;\ldots, p_t$; $q_1, q_2, \ldots ,q_t$; $r_1, r_2, \ldots, r_t$; and $s_1, s_2, \ldots, s_t$ be distinct primes such that $(p_i, q_i, r_i, s_i)\equiv (1,5,3,7) \pmod 8$ for $1 \leq i \leq t$. Consider $\mu \in \{1,-1\}$. Assume that
    \begin{enumerate}
        \item[(a)] $\legendre{p_i}{q_i}=\legendre{r_i}{s_i}=-1, \;\;\;\; \legendre{p_i}{s_i}=\legendre{p_i}{r_i}=\legendre{q_i}{r_i}=\legendre{q_i}{s_i}=1$ for all $i$,
        \item[(b)] $\legendre{p_i}{q_j}=\legendre{p_i}{r_j}=\legendre{p_i}{s_j}=\legendre{q_i}{r_j}=\legendre{q_i}{s_j}=\legendre{r_i}{s_j}=1$ for all $i\neq j$,
        \item[(c)] $\legendre{s_i}{s_j}=\legendre{r_i}{r_j}=\mu$ for all $i < j$, $\legendre{p_i}{p_j}=\legendre{q_i}{q_j}=1$ for all $i \neq j$.
        %$\legendre{s_i}{s_j}=\legendre{q_i}{q_j}=\mu$ for all $i < j$, $\legendre{p_i}{p_j}=\legendre{r_i}{r_j}=1$, for all $i \neq j$.
    \end{enumerate}
  Then $n = 2\prod\limits_{i=1}^{t} p_iq_ir_is_i$ is not a congruent number.
\end{theorem}

\section{Preliminaries}\label{preli}
In this section, we introduce the $2$-Selmer rank associated with an elliptic curve $E_{n}: y^2=x^3-n^2x$, and we examine the relationship between the $2$-Selmer rank and the rank of the Monsky matrix corresponding to a square-free positive integer $n$. 

Consider the congruent number elliptic curve $E_n$ for $n$ is a positive square-free integer. It is well-established that the subgroup of torsion points of the Mordel-Weil group $E_{n}(\Q)$ is given by
$$E_{n}(\Q)_{\operatorname{tors}}:= \{ \infty, \; (0,0)\; (n,0)\; (-n,0)\} \cong \Z/2\Z \oplus \Z/2\Z.$$  Consequently,  $E_n(\Q) \cong \Z^{r(n)} \oplus \Z/2\Z \oplus \Z/2\Z$. The quantity $r(n)$ is a non-negative integer and it is widely known as Mordell-Weil rank or algebraic rank of $E_n$ over $\Q$. Consider the $2$-Selmer group $\Sel_2(E_n/\Q)$ and the Tate-Shafarevich group $\Sh(E_n/\Q)$ as defined in \cite[Chapter 10]{Silverman} which fits into the following exact sequence: 
\begin{equation*}\label{exactenq}
	0 \to E_n(\Q)/2E_n\Q) \to \Sel_2(E_n/\Q) \to \Sh(E_n/\Q)[2] \to 0. 
\end{equation*} 
Note that $E_n(\Q)/2E_n(\Q) \cong (\Z/2\Z)^{r(n)+2}$. Therefore, $\Sel_2(E_n/\Q)$ is a $2$-group and its cardinality is divisible by $4$, on the account of torsion points of $E_{n}(\Q)$.  If we write the cardinality of the $2$-Selmer group as $\#\Sel_2(E_n/\Q)=2^{s(n)+2}$, then we have the inequality
\begin{equation}\label{rnsninequality}
    0 \leq r(n) \leq s(n).
\end{equation}
The non-negative integer $s(n)$ is referred as $2$-Selmer rank of $E_n$. Selmer conjecture predicts that the Mordell-Weil rank $r(n)$ and the $2$-Selmer rank $s(n)$ have the same parity. By the work of Monsky, we relate $s(n)$ with the rank of the Monsky matrix which is defined below. For more details, see \cite[Appendix]{hb}.

Let $p_1,p_2,\ldots ,p_m$ be distinct odd prime numbers and  $n=2^{\delta}p_1p_2\cdots p_m$ be a square-free positive integer, where $\delta \in \{0, 1\}$. Consider the map $\phi: \{-1,1\} \to \{0,1\}$ defined by $(-1)^\epsilon \mapsto \epsilon$. For each $l \in \{-1,2,-2\}$, we define a $m \times m$ diagonal matrix ${\mathrm{D}}_l=[d_{i,l}]$  and another  $m \times m$ matrix ${\mathrm{E}}=[e_{i,j}]$ respectively as follows: \begin{equation}\label{matrix D and E}
    d_{i,l}=\begin{cases}
    0	\text{ if } \legendre{l}{p_i}=1, \\
    1	\text{ if } \legendre{l}{p_i}=-1, \end{cases} e_{i,j}=\begin{cases}
    0	\text{ if } \legendre{p_j}{p_i}=1 \textit{ with } j\neq i,\\
    1	\text{ if } \legendre{p_j}{p_i}=-1 \textit{ with } j\neq i, \end{cases} \text{ and } e_{i,i}=\sum_{j; j\neq i} e_{i,j}. \end{equation}

The monsky matrices \cite[page 366]{hb} are given by
\begin{equation}\label{monskymat}
    {\mathrm{M}}_o=\begin{bmatrix}
\begin{array}{c c}
{\mathrm{D}}_2 & {\mathrm{E}}+{\mathrm{D}}_2 \\
{\mathrm{E}}+{\mathrm{D}}_{-2} & {\mathrm{D}}_{2}
\end{array}
\end{bmatrix}, \quad {\mathrm{M}}_e=\begin{bmatrix}
\begin{array}{c c}
{\mathrm{D}}_2 & {\mathrm{E}}+{\mathrm{D}}_2 \\
{\mathrm{E}}^{\top}+{\mathrm{D}}_{2} & {\mathrm{D}}_{-1}
\end{array}
\end{bmatrix}
\end{equation}
Here ${\mathrm{E}}^{\top}$ denotes the transpose of the matrix ${\mathrm{E}}$. Then the $2$-Selmer rank is given by 
\begin{align*}
    s(n)= \begin{cases}
        2m-\operatorname{rank}_{\F_2}{(\mathrm{M}}_o), \quad \text{ if } n \text{ is odd }, \\
        2m-\operatorname{rank}_{\F_2}{(\mathrm{M}}_e), \quad \text{ if } n \text{ is even } .
    \end{cases}
\end{align*}
where $\mathbb{F}_2$ is the finite field with two elements. Observe that both ${\mathrm{M}}_o$ and ${\mathrm{M}}_e$ are $2m \times 2m$ matrices. To prove a number $n$ to be a non-congruent number, we must show the Mordell-Weil rank $r(n)=0$. Hence it is enough prove that $s(n)=0$ by \eqref{rnsninequality}. More precisely, it is sufficient to show that the corresponding Monsky matrix ${\mathrm{M}}_o$ (resp. ${\mathrm{M}}_e$) has full rank $2m$ over $\F_2$ that is invertible over $\F_2$ if $n$ is odd (resp. $n$ is even).

\section{Some Preparatory Lemmas}\label{some pre lemmas}
In this section, we fix some notations that will be used throughout the rest of the paper. We then prove some basic relations satisfied by matrices over the finite field $\mathbb{F}_2$, which play a crucial role to prove the main results in \S \ref{sott}. Before that, we recall some properties of block matrices from \cite[Page 467 and 475]{Meyer} which will be used to compute the determinant of Monsky matrices.
\begin{lemma}
\label{determinant}
If $\mathbf{A}$ and $\mathbf{D}$ are square matrices then 
\begin{align*}
\det \left( \left[
\begin{array}{c|c}
\mathbf{A} & \mathbf{B} \\
\hline
\mathbf{O} & \mathbf{D}
\end{array}
\right] \right)= \det( \mathbf{A})\det(\mathbf{D})=\det \left( \left[
\begin{array}{c|c}
\mathbf{A} & \mathbf{O} \\
\hline
\mathbf{C} & \mathbf{D}
\end{array}
\right] \right).
\end{align*}
\end{lemma}

\begin{lemma}
\label{schur}
If $\mathbf{A}$ and $\mathbf{D}$ are square matrices, then 
\begin{align*}
\det \left( \left[
\begin{array}{c|c}
\mathbf{A} & \textbf{B} \\
\hline
\mathbf{C} & \mathbf{D}
\end{array}
\right] \right)= \begin{cases}
   \det( \mathbf{A})\det(\mathbf{D}-\mathbf{C}\mathbf{A}^{-1}\mathbf{B}) \quad \textit{ if } \mathbf{A}^{-1} exists,\\
   \det( \mathbf{D})\det(\mathbf{A}-\mathbf{B}\mathbf{D}^{-1}\mathbf{C}) \quad \textit{ if } \mathbf{D}^{-1} exists.
\end{cases}
\end{align*}
\end{lemma}

\noindent Lemma \ref{schur} is referred to as the {\it Schur Complement Lemma}.

\noindent Throughout this section, we use the following notation.
\begin{enumerate}\label{notation}
   %A(i,j) &: (i,j) \mbox{-th entry of  matrix } A,\\
   \item $R_m$ :~~ $m$-th row of a matrix
   \item $C_m$ : ~~$m$-th column of a matrix
 \item   ${\mathrm{O}} :\quad t\times t $    zero matrix,
 %\item   ${\mathrm{O}}_{i \times j} : i\times j $    zero matrix,
 \item  ${\mathrm{I}} :   \quad t\times t$ identity matrix,
  %\mathbf{O}_{k_1 \times k_2}&: \quad  k_1 \times k_2 \mbox{ zero matrix},\\
  %\mathbf{Z}_{\beta}&: \quad  k \times k \mbox{ matrix whose diagonal entries are } \beta \mbox{ and other entries are 1,}\\
  \item ${\mathrm{N}} : \quad  t \times t$  matrix having all the entries $1$,
  
  %\item ${\mathrm{R}} : \quad  t \times t \mbox{ matrix satisfying}  \quad {\mathrm{R}}= {\mathrm{N}}-{\mathrm{I}}$,

   %\item ${\mathrm{N}}_{i \times j} : i \times j$  matrix having all the entries $1$,

 \item  ${\mathrm{U}}_{f} : \quad t \times t$  \ upper-triangular matrix with $i$-th diagonal entries  $(t-i)+f$, and all upper-diagonal entries  1,
  
  \item  ${\mathrm{L}}_{f} : \quad t \times t$  lower-triangular matrix with $i$-th diagonal entries $f+(i-1)$, and all lower-diagonal entries 1.
  \item Let $l_{1}, l_{2}, \ldots, l_{t}$ be $t$ distinct primes with $l_{i} \equiv 3 \pmod{4}$ for each $i = 1, 2, \ldots, t$. Then ${\mathrm{T}}_{l,f}$ denotes a $t \times t$ matrix defined by
   ${\mathrm{T}}_{l,f}=\begin{cases}
    {\mathrm{L}}_{f} \quad \text{ if } \legendre{l_i}{l_j}=-1 \text{ for all } i<j,\\
    {\mathrm{U}}_{f}\quad  \text{ if } \legendre{l_i}{l_j}=1 \text{ for all } i<j.\\
    \end{cases}$

    \item Fix $\alpha_{0}, \beta_{0}  \in \{ 1, 3, 5,7 \}$.  Let $l_{1}, l_{2}, \ldots, l_{t}$ and $k_{1}, k_{2}, \ldots, k_{t}$ be distinct primes such that $l_i \equiv \alpha_{0} \pmod{8}$ and $k_i \equiv \beta_{0} \pmod{8}$ for all $i$. Define ${\mathrm{D}}_{lk}$ to be the $t \times t$ diagonal matrix with diagonal entries given by ${\mathrm{D}}_{lk}(i,i) = \phi\left(\legendre{k_i}{l_i}\right).$

 \item For a matrix $\mathrm{M}$, let $\mathrm{M}(i,j)$ denote its $(i,j)$-th entry. For integers $a, b \in \mathbb{Z}$, we write $a \equiv_{2} b$ to mean $a \equiv b \pmod{2}$. Similarly, for matrices $A$ and $B$ of the same size, the notation $A \equiv_{2} B$ signifies that $A(i,j) \equiv B(i,j) \pmod{2}$ for all $i,j$.

    \end{enumerate}
    
%For any matrix, we denote $m$-th row by $R_m$ and  $m$-th column by $C_m$.
The following lemmas easily follow from the definition. 
%\begin{lemma}\label{tlf}
 %  Let ${\mathrm{T}}_{l,0}$ be given as above. Then we have ${\mathrm{T}}_{l,0}^2 \equiv {\mathrm{T}}_{l,0} \pmod 2$.
%\end{lemma}

%\begin{lemma}\label{RT}
 %   Let $t$ be odd. Then ${\mathrm{N}}_{1\times t} {\mathrm{T}}_{l,0}\equiv {\mathrm{O}}_{1 \times t} \pmod2$ and $ {\mathrm{T}}_{l,0} {\mathrm{N}}_{t\times 1} \equiv {\mathrm{O}}_{t \times 1} \pmod 2$.
%\end{lemma}

\begin{lemma}\label{itemize}
Notations are as above. Then we have the followings: 
\begin{itemize}
    \item[(a)] ${\mathrm{N}}^2 \equiv_{2} t\mathrm{N}$.
    \item[(b)] ${\mathrm{T}}_{l,0}^{\top}+{\mathrm{T}}_{l,0} \equiv_{2} \mathrm{N}+ \mathrm{I}$.

    \item[(c)] ${\mathrm{T}}_{l,0}^{\top}{\mathrm{T}}_{l,0} \equiv_{2} (t-1){\mathrm{N}} $.
    \item[(d)] ${\mathrm{T}}_{l,0}{\mathrm{T}}_{l,0}^{\top} \equiv_{2} \mathrm{O} $.
    \item[(e)]  ${\mathrm{T}}_{l,t}^{\top}{\mathrm{T}}_{l,t} \equiv_{2}  {\mathrm{N}} $. 
    \item[(f)] ${\mathrm{T}}_{l,0}^{\top}{\mathrm{N}} \equiv_{2} (t-1) {\mathrm{N}} $. 
    \item[(g)] ${\mathrm{T}}_{l,0}{\mathrm{N}} \equiv_{2}  {\mathrm{O}} $.
    \item[(h)] ${\mathrm{T}}_{l,0}^2 \equiv_{2} {\mathrm{T}}_{l,0} $.
    \item[(i)] For $t$ be odd, we have ${\mathrm{N}}_{1\times t} {\mathrm{T}}_{l,0}\equiv_{2} {\mathrm{O}}_{1 \times t}$ and $ {\mathrm{T}}_{l,0} {\mathrm{N}}_{t\times 1} \equiv_{2} {\mathrm{O}}_{t \times 1} $.
\end{itemize}
\end{lemma}

\begin{proof}
    We prove each point separately. We will give a proof for the case ${\mathrm{T}}_{l,0}=\mathrm{U}_0$ only since the case ${\mathrm{T}}_{l,0}=\mathrm{L}_0$ follows similarly.
    \begin{itemize}
        \item[(a)] If we denote the entries of $\mathrm{N}$ by $\mathrm{N}(i,j)$, then $\mathrm{N}(i,j)=1$. By the definition of multiplication of matrices, the entries $\mathrm{N}^2(i,j)= \sum_{k=1}^t \mathrm{N}(i,k) \mathrm{N}(k, j)= t$. Hence we have the desired result.

        \item[(b)] Let ${\mathrm{T}}_{l,0}=\mathrm{U}_0$ where ${\mathrm{U}}_0$ is defined as above. Then ${\mathrm{U}}^{\top}$ is a lower triangular matrix with the diagonal entries to be $t-i$ and all the lower-diagonal entries are 1. Since the diagonal entries of a matrix and its transpose remain same, modulo 2, the diagonal entries of ${\mathrm{T}}_{l,0}^{\top}+{\mathrm{T}}_{l,0}$ is $0$, and all the non-diagonal entries are {1}. Hence ${\mathrm{T}}_{l,0}^{\top}+{\mathrm{T}}_{l,0} \equiv_{2} {\mathrm{N}} + {\mathrm{I}} $. %In a similar way, we can proof the same when ${\mathrm{T}}_{l,0}=\mathrm{L}_0$.
        
        \item[(c)] Let ${\mathrm{T}}_{l,0}=\mathrm{U}_0$. Then ${\mathrm{T}}_{l,0}^{\top}{\mathrm{T}}_{l,0}=\mathrm{U}_0^{\top}\mathrm{U}_0$ and {\scriptsize{\begin{align*}
        \mathrm{U}_0^{\top}\mathrm{U}_0(i,j) &= 
        \begin{cases}
            (j-1)+(t-j) \equiv_{2} t-1  & \text{if } i>j,\\
            (i-1)+(t-i)^2 \equiv_{2} t-1 & \text{if } i=j,\\
            (i-1)+(t-i) \equiv_{2} t-1 & \text{if }  i<j  .
        \end{cases} %\\
       % & \equiv t-1 \pmod 2.
    \end{align*}}}
%The case where ${\mathrm{T}}_{l,0}=\mathrm{L}_0$ can be proved similarly.
        \item[(d)] %We only prove the case where ${\mathrm{T}}_{l,0}=\mathrm{U}_0$ since the case ${\mathrm{T}}_{l,0}=\mathrm{L}_0$ follows similarly.
        Assume that ${\mathrm{T}}_{l,0}=\mathrm{U}_0$. %The proof is pretty straight forward since
        Hence {\scriptsize{$\mathrm{U}_0^{\top}\mathrm{U}_0(i,j)= \begin{cases}
            (t-i)+ \underbrace{1+1+\cdots+1}_{t-i \text{ times}} \equiv_{2} 0 \ & \text{if } i>j,\\
            (t-i)^2+ \underbrace{1+1+\cdots+1}_{t-i \text{ times}} \equiv_{2} 0 & \text{if } i=j,\\
            (t-j)+ \underbrace{1+1+\cdots+1}_{t-j \text{ times}} \equiv_{2} 0  & \text{if }  i<j,  
        \end{cases}
        $}}\\
that is, for $1 \leq i,j \leq t$, $\mathrm{U}_0^{\top}\mathrm{U}_0(i,j)\equiv_{2} 0 $. Hence we have the desired equality.

        \item[(e)] Note that ${\mathrm{T}}_{l,t}={\mathrm{T}}_{l,0}+t\mathrm{I}$ and ${\mathrm{T}}_{l,t}^{\top}={\mathrm{T}}_{l,0}^\top+t\mathrm{I}$. Hence 
        \begin{eqnarray*}
        {\mathrm{T}}_{l,t}^{\top}{\mathrm{T}}_{l,t} &=& ({\mathrm{T}}_{l,0}^{\top} + t\mathrm{I})({\mathrm{T}}_{l,0} + t\mathrm{I})  \\ 
        & \equiv_2 & {\mathrm{T}}_{l,0}^{\top}{\mathrm{T}}_{l,0} + t ({\mathrm{T}}_{l,0}^{\top}+{\mathrm{T}}_{l,0}) + t\mathrm{I}  \\
        & \overset{\text{by} (b), (c)}{\equiv_2} &  (t-1)\mathrm{N} + t(\mathrm{N}+\mathrm{I}) + t\mathrm{I}  \\
        & \equiv_2 & \mathrm{N} .
        \end{eqnarray*}
        
        \item[(f)] Let ${\mathrm{T}}_{l,0} = \mathrm{U}_0$. Then by the definition of $\mathrm{U}_0^{\top}$ as in (b), $\mathrm{U}_0^{\top}\mathrm{N}(i,j)$ equal to sum of entries of $i$-th row which is equal to $(i-1)+(t-i)=t-1$ which gives that $\mathrm{U}_0^{\top}\mathrm{N}\equiv_{2}(t-1)\mathrm{N}$. %In a similar way, we can prove that the same equality if ${\mathrm{T}}_{l,0} = \mathrm{L}_0$.

        \item[(g)] Assume that ${\mathrm{T}}_{l,0} = \mathrm{U}_0$. Clearly the $(i,j)$-th entry of $\mathrm{U}_0\mathrm{N}$ is the sum of the entries of the $i$-th row which is equal to $(t-i)+ \underbrace{1+1+\cdots+1}_{t-i \text{ times}} \equiv_{2} 0 $. Hence is the result.

        \item[(h)] We first assume that the matrix ${\mathrm{T}}_{l,0}$ is equal to $\mathrm{U}_0$. Then by definition 
       {\scriptsize{ \begin{align*}{\mathrm{U}}_{0}^2(i,j)=  \sum_{k=1}^t{\mathrm{U}}_{0}(i,k){\mathrm{U}}_{0}(k,j)= \begin{cases}
                0 & \text{if } i>j, \\
                (t-i)^2 \equiv_{2} t-i & \text{if }  i=j, \\
                (t-i)+(t-j)+(j-i-1) \equiv_{2} 1  & \text{if } i<j.
             \end{cases}
        \end{align*} 
        }}
    that is ${\mathrm{U}}_{0}^2(i,j)\equiv_{2} {\mathrm{U}}_{0}(i,j) $, hence we have the desired result. %The case where ${\mathrm{T}}_{l,0}=\mathrm{L}_0$ can be proved similarly.

    \item[(i)] Let $t$ be odd and ${\mathrm{T}}_{l,0}=\mathrm{U}_0$. Then ${\mathrm{N}}_{1\times t} {\mathrm{U}}_{0}$ is a $1 \times t$ row matrix and ${\mathrm{N}}_{1\times t} {\mathrm{U}}_{0}(1,j)= \text{sum of the entries of the $j$-th columns of } \mathrm{U}_{0}= (t-j)+ \underbrace{1+1+\cdots+1}_{j-1 \text{ times}} = t-1 \equiv_{2} 0 $. Similarly, ${\mathrm{U}}_{0} {\mathrm{N}}_{t\times 1}$ is a $t\times 1$ column matrix and ${\mathrm{U}}_{0} {\mathrm{N}}_{t\times 1}(i,1)=$ sum of the entries of the $i$-th row of  $\mathrm{U}_{0} \equiv_{2} 0 $ (since $t$ is odd). 
    \end{itemize}
\end{proof}

\begin{lemma}\label{UL-LU}
Let $\mathrm{U}_0$ and $\mathrm{L}_0$ be upper and lower triangular matrices of size $t \times t$, respectively, as in the notation. Then:
\begin{itemize}
    \item[(a)] The matrix $\mathrm{L}_0^{\top}\mathrm{U}_0$ is upper triangular, and
    \begin{align*}
        \mathrm{L}_0^{\top}\mathrm{U}_0(i,j) &=
        \begin{cases}
            0 & \text{if } i>j,\\
            (t-i)(i-1) & \text{if } i=j,\\
            t-2 & \text{if } i<j.
        \end{cases}
    \end{align*}

    \item[(b)] The matrix $\mathrm{U}_0^{\top}\mathrm{L}_0$ is lower triangular, and
    \begin{align*}
        \mathrm{U}_0^{\top}\mathrm{L}_0(i,j) &=
        \begin{cases}
            0 & \text{if } i<j,\\
            (t-i)(i-1) & \text{if } i=j,\\
            t-2 & \text{if } i>j.
        \end{cases}
    \end{align*}
    \item[(c)] $\mathrm{U}_{0}^{\top}+\mathrm{L}_{0} \equiv_{2} \mathrm{L}_{0}^{\top}+\mathrm{U}_{0} \equiv_{2} (t+1) \mathrm{I} $.
\end{itemize}
\end{lemma}

\begin{proof}
    We prove each part separately.
    \begin{itemize}
    \item[(a)] Note that {\scriptsize{$\mathrm{L}_0^{\top} =
        \begin{cases}
            0 & \text{if } i>j,\\
            (i-1) & \text{if } i=j,\\
            1 & \text{if } i<j.
        \end{cases} $}} and {\scriptsize{$\mathrm{U}_0 = 
        \begin{cases}
            0 & \text{if } i>j,\\
            (t-i) & \text{if } i=j,\\
            1 & \text{if } i<j,
        \end{cases} $}} that is, both $\mathrm{L}_0^{\top}$ and $\mathrm{U}_0$ are upper triangular matrices and hence is their product. Furthermore, {\scriptsize{\begin{align*}
        \mathrm{L}_0^{\top}\mathrm{U}_0(i,j) &= 
        \begin{cases}
            0 & \text{if } i>j,\\
            (i-1)(t-i) & \text{if } i=j,\\
            (i-1)+(t-j)+\underbrace{1+1+\cdots+1}_{(j-i-1) \text{ times}}=t-2 & \text{if } i<j.
        \end{cases}
    \end{align*}
        }}
        \item[(b)] Using similar argument as in the previous case, we get the desired value of $\mathrm{U}_0^\top\mathrm{L}_0(i,j)$.

        \item[(c)]This follows from the definition. Indeed, both $\mathrm{U}_0^{\top}$ and $\mathrm{L}_0$ are lower triangular matrices with diagonal entries to be $t-i$ and $i-1$ respectively, and lower diagonal entries to be $1$ (for both matrices). This implies that the diagonal entries of $\mathrm{U}_0^{\top}+\mathrm{L}_0$ is $t-1 \equiv_{2} t+1$ and the non-diagonal entries are $0$. Hence we have the desired result. Similarly we have $\mathrm{L}_0^{\top}+\mathrm{U}_0 \equiv_{2} (t+1)\mathrm{I} $. 
    \end{itemize}
\end{proof}

\section{Proof of the theorems}\label{proof of the theorems}
In this section, we prove the theorems stated in \S \ref{sott}. In each case, we show that the $2$-Selmer rank of the elliptic curve $E_n$ is $0$ by showing that the corresponding Monsky matrix is invertible over $\F_2$. Our computations rely on elementary row and column operations on the matrix, noting that the rank remains invariant under such operations by basic linear algebra.   %For a simpler approach to the proofs, an integer of the form $n=\prod\limits_{i=1}^{t}p_i q_i r_i\quad \text{(or } n=2 \prod\limits_{i=1}^{t}p_i q_i r_i\text{)}$ may be viewed as  $n=\prod\limits_{i=1}^{t} p_i\prod\limits_{i=1}^{t} q_i\prod\limits_{i=1}^{t} r_i\quad  \text{(or } n=2 \prod\limits_{i=1}^{t} p_i\prod\limits_{i=1}^{t} q_i\prod\limits_{i=1}^{t} r_i\text{)},$ as in Theorems~\ref{157}, \ref{355} (or Theorem~\ref{533}).
\begin{proof}[Proof of Theorem \ref{157}]
Let $ n=\prod_{i=1}^{t} p_i q_i r_i $ be as in Theorem~\ref{157}. For simpler approach in the proof, we regard it as
$ n=\left(\prod_{i=1}^{t} p_i\right)\left(\prod_{i=1}^{t} q_i\right)\left(\prod_{i=1}^{t} r_i\right) $.
Since $q_i \equiv 5 \pmod{8}$, it follows that $\legendre{q_i}{r_i} = \legendre{r_i}{q_i}$ for all $1 \le i \le t$, and hence ${\mathrm{D}}_{qr} = {\mathrm{D}}_{rq}$.   By (\ref{matrix D and E}), We have 
{\scriptsize{
\begin{equation*}
 \label{D157} 
{\mathrm{D}}_2=\left[
\begin{array}{c|c|c}
{\mathrm{O}} & {\mathrm{O}} & {\mathrm{O}} \\
\hline
{\mathrm{O}} & {\mathrm{I}} & {\mathrm{O}} \\
\hline
{\mathrm{O}} & {\mathrm{O}} & {\mathrm{O}}
\end{array}
\right],\quad {\mathrm{D}}_{-2}=\left[
\begin{array}{c|c|c}
{\mathrm{O}} & {\mathrm{O}} & {\mathrm{O}} \\
\hline
{\mathrm{O}} & {\mathrm{I}} & {\mathrm{O}} \\
\hline
{\mathrm{O}} & {\mathrm{O}} & {\mathrm{I}}
\end{array}
\right], \quad {\mathrm{E}} \equiv_{2}
\left[
\begin{array}{c|c|c}
{\mathrm{O}} & {\mathrm{I}} & {\mathrm{I}} \\
\hline
{\mathrm{I}} & {\mathrm{I}}+{\mathrm{D}}_{qr} & {\mathrm{D}}_{qr}\\
\hline
{\mathrm{I}} & {\mathrm{D}}_{qr} & {\mathrm{I}} + {\mathrm{D}}_{qr} + {\mathrm{T}}_{r,0}
\end{array}
\right].     
\end{equation*}
}}
Using the definition of Monsky matrix as in \eqref{monskymat}, we have 
{\scriptsize{\begin{equation*}
 {\mathrm{M}}_o \equiv_{2}\left[
\begin{array}{c|c|c|c|c|c}
{\mathrm{O}} & {\mathrm{O}} & {\mathrm{O}} & {\mathrm{O}} & {\mathrm{I}} & {\mathrm{I}}\\
\hline
{\mathrm{O}} & {\mathrm{I}} & {\mathrm{O}} & {\mathrm{I}} & {\mathrm{D}}_{qr} & {\mathrm{D}}_{qr} \\
\hline
{\mathrm{O}} & {\mathrm{O}} & {\mathrm{O}} & {\mathrm{I}} & {\mathrm{D}}_{qr} & {\mathrm{I}}+{\mathrm{D}}_{qr}+{\mathrm{T}}_{r,0}\\
\hline
{\mathrm{O}} & {\mathrm{I}} & {\mathrm{I}} & {\mathrm{O}} & {\mathrm{O}} & {\mathrm{O}}\\
\hline
{\mathrm{I}} & {\mathrm{D}}_{qr} & {\mathrm{D}}_{qr} & {\mathrm{O}} & {\mathrm{I}} & {\mathrm{O}}\\
\hline
{\mathrm{I}} & {\mathrm{D}}_{qr} & {\mathrm{D}}_{qr}+ {\mathrm{T}}_{r,0} & {\mathrm{O}} & {\mathrm{O}} & {\mathrm{O}}\\
\end{array}
\right]. 
\end{equation*}
}}
Next, for each $i \in \{1,2,\ldots,t\}$, we interchange the rows $R_i$ and $R_{5t+i}$ of $\mathrm{M}_{o}$. In short, we interchange the rows in the first block with those in the sixth block of $\mathrm{M}_{o}$, and obtain
{ \scriptsize{
\begin{equation*}
 {\mathrm{M}}^{\prime}_o =
 \left[
\begin{array}{c|c|c|c|c|c}
{\mathrm{I}} & {\mathrm{D}}_{qr} & {\mathrm{D}}_{qr}+ {\mathrm{T}}_{r,0} & {\mathrm{O}} & {\mathrm{O}} & {\mathrm{O}}\\
\hline
{\mathrm{O}} & {\mathrm{I}} & {\mathrm{O}} & {\mathrm{I}} & {\mathrm{D}}_{qr} & {\mathrm{D}}_{qr} \\
\hline
{\mathrm{O}} & {\mathrm{O}} & {\mathrm{O}} & {\mathrm{I}} & {\mathrm{D}}_{qr} & {\mathrm{I}}+{\mathrm{D}}_{qr}+{\mathrm{T}}_{r,0}\\
\hline
{\mathrm{O}} & {\mathrm{I}} & {\mathrm{I}} & {\mathrm{O}} & {\mathrm{O}} & {\mathrm{O}}\\
\hline
{\mathrm{I}} & {\mathrm{D}}_{qr} & {\mathrm{D}}_{qr} & {\mathrm{O}} & {\mathrm{I}} & {\mathrm{O}}\\
\hline
{\mathrm{O}} & {\mathrm{O}} & {\mathrm{O}} & {\mathrm{O}} & {\mathrm{I}} & {\mathrm{I}}\\
\end{array}
\right]. 
\end{equation*}
}}
Now let
{\scriptsize{
\begin{equation*}
\label{ABCD157} 
{\mathrm{A}}=\left[
\begin{array}{c|c|c|c}
{\mathrm{I}} & {\mathrm{D}}_{qr} &  {\mathrm{D}}_{qr}+{\mathrm{T}}_{r,0} & {\mathrm{O}}\\
\hline
{\mathrm{O}} & {\mathrm{I}} & {\mathrm{O}} & {\mathrm{I}} \\
\hline
{\mathrm{O}} & {\mathrm{O}} & {\mathrm{O}} & {\mathrm{I}} \\
\hline
{\mathrm{O}} & {\mathrm{I}} & {\mathrm{I}} & {\mathrm{O}}
\end{array}
\right],\quad {\mathrm{B}}=\left[
\begin{array}{c|c}
{\mathrm{O}} & {\mathrm{O}} \\
\hline
{\mathrm{D}}_{qr} & {\mathrm{D}}_{qr} \\
\hline
{\mathrm{D}}_{qr} & {\mathrm{I}}+{\mathrm{D}}_{qr}+ {\mathrm{T}}_{r,0} \\
\hline 
{\mathrm{O}} & {\mathrm{O}}
\end{array}
\right], 
{\mathrm{C}}= \left[
\begin{array}{c|c|c|c}
{\mathrm{I}} & {\mathrm{D}}_{qr} & {\mathrm{D}}_{qr} & {\mathrm{O}}\\
\hline
{\mathrm{O}} & {\mathrm{O}} & {\mathrm{O}} & {\mathrm{O}}
\end{array}
\right], \quad \text{ and, } {\mathrm{D}}= \left[\begin{array}{c|c}
{\mathrm{I}} & {\mathrm{O}}\\
\hline
{\mathrm{I}} & {\mathrm{I}}
\end{array}\right].
\end{equation*}
}}
 Note that Lemma \ref{determinant} implies that the matrix ${\mathrm{D}}$ is invertible. Applying Lemma \ref{schur}, we deduce that ${\mathrm{M}}^{\prime}_o$ is invertible over $\mathbb{F}_2$ if and only if the matrix
 {\scriptsize{
\begin{equation*}
\mathrm{M}^{\prime \prime }_{o}={\mathrm{A}}-{\mathrm{B}}{\mathrm{D}}^{-1}{\mathrm{C}} = \left[
\begin{array}{c|c|c|c}
{\mathrm{I}} & {\mathrm{D}}_{qr} &  {\mathrm{D}}_{qr}+{\mathrm{T}}_{r,0} & {\mathrm{O}}\\
\hline
{\mathrm{O}} & {\mathrm{I}} & {\mathrm{O}} & {\mathrm{I}} \\
\hline
{\mathrm{I}}+{\mathrm{T}}_{r,0} & {\mathrm{D}}_{qr}+{\mathrm{T}}_{r,0}{\mathrm{D}}_{qr} & {\mathrm{D}}_{qr}+{\mathrm{T}}_{r,0}{\mathrm{D}}_{qr} & {\mathrm{I}} \\

\hline
{\mathrm{O}} & {\mathrm{I}} & {\mathrm{I}} & {\mathrm{O}}
\end{array}\right]  
\end{equation*}
}}
is invertible over $\mathbb{F}_2$. For each $1 \leq i \leq t$, replace the column $C_{2t+i}$ by $C_{t+i} + C_{2t+i}$ in $\mathrm{M}^{\prime \prime}_{o}$; equivalently, replace the third block of columns by the sum of the second and third block columns of $\mathrm{M}^{\prime \prime }_{o}$, we obtain 
{\scriptsize{
\begin{equation*}
\mathrm{M}^{\prime \prime} = \left[
\begin{array}{c|c|c|c}
{\mathrm{I}} & {\mathrm{D}}_{qr} & {\mathrm{T}}_{r,0} & {\mathrm{O}}\\
\hline
{\mathrm{O}} & {\mathrm{I}} & {\mathrm{I}} & {\mathrm{I}} \\
\hline
{\mathrm{I}}+{\mathrm{T}}_{r,0} & {\mathrm{D}}_{qr}+{\mathrm{T}}_{r,0}{\mathrm{D}}_{qr} & {\mathrm{O}} & {\mathrm{I}} \\
\hline
{\mathrm{O}} & {\mathrm{I}} & {\mathrm{O}} & {\mathrm{O}}
\end{array}\right]. 
\end{equation*}
}}
 By Lemma \ref{determinant}, we see that {\scriptsize{$\left[\begin{array}{c|c}
\mathrm{I} & \mathrm{D}_{qr} \\
\hline
\mathrm{O} & \mathrm{I}
\end{array}\right]$}} is invertible over $\mathbb{F}_2$, and 
{\scriptsize{$\left[\begin{array}{c|c}
\mathrm{I} & \mathrm{D}_{qr} \\
\hline
\mathrm{O} & \mathrm{I}
\end{array}\right]^{-1}
\equiv_{2}
\left[\begin{array}{c|c}
\mathrm{I} & \mathrm{D}_{qr} \\
\hline
\mathrm{O} & \mathrm{I}
\end{array}\right]$}}. Applying Lemma \ref{schur} to $\mathrm{M}^{\prime\prime}$,
we deduce that
{\scriptsize{
\begin{equation*}
\det \mathrm{M}^{\prime \prime} \equiv_2  \det \left[\begin{array}{c|c}
  {\mathrm{T}}_{r,0}+{\mathrm{T}}_{r,0}^2   &  {\mathrm{I}}\\
  \hline {\mathrm{I}} & {\mathrm{I}}
\end{array}\right].  
\end{equation*}
}}
Using part (h) of Lemma \ref{itemize}, we obtain ${\mathrm{T}}_{r,0} + {\mathrm{T}}_{r,0}^2 \equiv_{2} \mathrm{O}$. Therefore, by applying Lemma \ref{schur} to $\mathrm{M}^{\prime \prime}$, we have $\det \mathrm{M}^{\prime \prime} \equiv_{2} \det \mathrm{I}$. Hence $\mathrm{M}'_{o}$ is invertible, and since $\operatorname{rank}_{\F_2}{(\mathrm{M}_{o})} = \operatorname{rank}_{\F_2}{(\mathrm{M}'_{o})}$, it follows that $\mathrm{M}_{o}$ has full rank. This completes 
the proof.
\end{proof}

\begin{proof}[Proof of Theorem \ref{355}]
    Let $(p_i, q_i, r_i) \equiv (3,5,5) \pmod{8}$ be distinct primes as in Theorem \ref{355}, and set $n = \prod_{i=1}^{t} p_i q_i r_i$. For simpler approach in the proof, we regard it as
$ n=\left(\prod_{i=1}^{t} p_i\right)\left(\prod_{i=1}^{t} q_i\right)\left(\prod_{i=1}^{t} r_i\right) $. Note that,  we have $D_{qr}=D_{rq}$. Here Since
\begin{equation*}
\legendre{-1}{p} =
\begin{cases}
1 & \text{if } p \equiv 1 \pmod{4}, \\
-1 & \text{if } p \equiv 3 \pmod{4},
\end{cases}   \text{ and } \legendre{2}{p} = -1 \text{ whenever } p \equiv  \pm  3  \pmod{8},
\end{equation*}
it follows that 
{\scriptsize{
\begin{equation*}
 \label{D355} 
{\mathrm{D}}_2=\left[
\begin{array}{c|c|c}
{\mathrm{I}} & {\mathrm{O}} & {\mathrm{O}} \\
\hline
{\mathrm{O}} & {\mathrm{I}} & {\mathrm{O}} \\
\hline
{\mathrm{O}} & {\mathrm{O}} & {\mathrm{I}}
\end{array}
\right],\quad {\mathrm{D}}_{-2}=\left[
\begin{array}{c|c|c}
{\mathrm{O}} & {\mathrm{O}} & {\mathrm{O}} \\
\hline
{\mathrm{O}} & {\mathrm{I}} & {\mathrm{O}} \\
\hline
{\mathrm{O}} & {\mathrm{O}} & {\mathrm{I}}
\end{array}
\right], \quad {\mathrm{E}} \equiv_{2} \left[
\begin{array}{c|c|c}
{\mathrm{T}}_{p,0} & {\mathrm{I}} & {\mathrm{I}} \\
\hline
{\mathrm{I}} & {\mathrm{I}}+{\mathrm{D}}_{qr} & {\mathrm{D}}_{qr}\\
\hline
{\mathrm{I}} & {\mathrm{D}}_{qr} & {\mathrm{I}} + {\mathrm{D}}_{qr}
\end{array}
\right]     
\end{equation*}
}}
as in \eqref{matrix D and E}. Using the definition of Monsky matrix as in (\ref{monskymat}), we have 
{\scriptsize{
\begin{equation*}
 {\mathrm{M}}_o \equiv_{2}\left[
\begin{array}{c|c|c|c|c|c}
{\mathrm{I}} & {\mathrm{O}} & {\mathrm{O}} & {\mathrm{I}}+{\mathrm{T}}_{p,0} & {\mathrm{I}} & {\mathrm{I}}\\
\hline
{\mathrm{O}} & {\mathrm{I}} & {\mathrm{O}} & {\mathrm{I}} & {\mathrm{D}}_{qr} & {\mathrm{D}}_{qr} \\
\hline
{\mathrm{O}} & {\mathrm{O}} & {\mathrm{I}} & {\mathrm{I}} & {\mathrm{D}}_{qr} & {\mathrm{D}}_{qr}\\
\hline
{\mathrm{T}}_{p,0} & {\mathrm{I}} & {\mathrm{I}} & {\mathrm{I}} & {\mathrm{O}} & {\mathrm{O}}\\
\hline
{\mathrm{I}} & {\mathrm{D}}_{qr} & {\mathrm{D}}_{qr} & {\mathrm{O}} & {\mathrm{I}} & {\mathrm{O}}\\
\hline
{\mathrm{I}} & {\mathrm{D}}_{qr} & {\mathrm{D}}_{qr} & {\mathrm{O}} & {\mathrm{O}} & {\mathrm{I}}\\
\end{array}
\right]. 
\end{equation*}
}}

Applying Lemma \ref{schur} and part (h) of Lemma \ref{itemize} to ${\mathrm{M}}_{o}$, we conclude that $\mathrm{M}_{o}$ is invertible over $\mathbb{F}_2$ if and only if
 
{\scriptsize{ 
 \begin{equation*}
 {\mathrm{M}}_o'=\left[
\begin{array}{c|c|c}
{\mathrm{I}} & {\mathrm{T}}_{p,0} & {\mathrm{T}}_{p,0} \\
\hline
{\mathrm{T}}_{p,0}+{\mathrm{I}} & {\mathrm{O}} & {\mathrm{I}} \\
\hline
{\mathrm{T}}_{p,0}+{\mathrm{I}} & {\mathrm{I}} & {\mathrm{O}}
\end{array}
\right] 
\end{equation*}
}}
is invertible over $\mathbb{F}_2$. We apply Lemma \ref{schur} and part (h) of Lemma \ref{itemize} to conclude that $\mathrm{M}'_{o}$ is invertible over $\mathbb{F}_2$.
\end{proof}

\begin{proof}[Proof of Theorem \ref{377}]
Let $p, q \equiv 7 \pmod{8}$ be distinct primes, and let $r_1, r_2, \ldots, r_t$ be distinct primes congruent to $3 \pmod{8}$. Assume that these primes satisfy the conditions of Theorem \ref{377} with $\alpha = 1$. Note that, $t$ is odd here. Then 
{\scriptsize{
\begin{equation*}
 %\label{D377} 
{\mathrm{D}}_2=\left[
\begin{array}{c|c|c}
0 & 0 & {\mathrm{O}}_{1 \times t} \\
\hline
0 & 0 & {\mathrm{O}}_{1 \times t} \\
\hline
{\mathrm{O}}_{t\times 1} & {\mathrm{O}}_{t\times 1} & {\mathrm{I}}
\end{array}
\right]_,\quad {\mathrm{D}}_{-2}=\left[
\begin{array}{c|c|c}
1 & 0 & {\mathrm{O}}_{1\times t} \\
\hline
0 & 1 & {\mathrm{O}}_{1 \times t} \\
\hline
{\mathrm{O}}_{t \times 1} & {\mathrm{O}}_{t \times 1} & {\mathrm{O}}
\end{array}
\right], \quad {\mathrm{E}}\equiv_{2}\left[
\begin{array}{c|c|c}
1 & 1 & {\mathrm{O}}_{1\times t} \\
\hline
0 & 1 & {\mathrm{N}}_{1\times t}\\
\hline
{\mathrm{N}}_{t \times 1} & {\mathrm{O}}_{t \times 1} & {\mathrm{I}} + {\mathrm{T}}_{r,0}
\end{array}
\right].     
\end{equation*} 
}}
  are $(t+2)\times (t+2)$ matrices. Here ${\mathrm{O}}_{1\times t}$ (resp. ${\mathrm{O}}_{t\times 1}$) is a $1\times t$  matrix (resp. $t\times 1$ matrix) with each entry $0$. Let ${\mathrm{N}}_{1\times t}$ (resp. ${\mathrm{N}}_{t\times 1}$) is a $1\times t$  matrix (resp. $t\times 1$ matrix) with each entry $1$. Then using the definition of Monsky matrix as in \eqref{monskymat}  we have
{\scriptsize{
\begin{equation*} \label{M577}
 {\mathrm{M}}_o \equiv_{2}\left[
\begin{array}{c|c|c|c|c|c}
0 & 0 & {\mathrm{O}}_{1\times t} & 1 & 1 & {\mathrm{O}}_{1 \times t}\\
\hline
0 & 0 & {\mathrm{O}}_{1 \times t} & 0 & 1 & {\mathrm{N}}_{1\times t} \\
\hline
{\mathrm{O}}_{t \times 1} & {\mathrm{O}}_{t \times 1} & {\mathrm{I}} & {\mathrm{N}}_{t \times 1} & {\mathrm{O}}_{t \times 1} & {\mathrm{T}}_{r,0}\\
\hline
0 & 1 & {\mathrm{O}}_{1 \times t} & 0 & 0 & {\mathrm{O}}_{1 \times t}\\
\hline
0 & 0 & {\mathrm{N}}_{1 \times t} & 0 & 0 & {\mathrm{O}}_{1 \times t}\\
\hline
{\mathrm{N}}_{t \times 1} & {\mathrm{O}}_{t \times 1} & {\mathrm{I}}+{\mathrm{T}}_{r,0} & {\mathrm{O}}_{t \times 1} & {\mathrm{O}}_{t \times 1} & {\mathrm{I}}\\
\end{array}
\right]. 
\end{equation*}
}}
Recall from part (i) of Lemma \ref{itemize}, we have ${\mathrm{N}}_{1\times t} {\mathrm{T}}_{r,0} \equiv_{2} {\mathrm{O}}_{1 \times t}$ and ${\mathrm{T}}_{r,0} {\mathrm{N}}_{t\times 1}  \equiv_{2} {\mathrm{O}}_{t \times 1}$ as $t$ is odd.
Using Lemma \ref{schur}, part (h) and (i) of Lemma \ref{itemize} together with  
{\scriptsize{
\begin{equation*}
{\mathrm{A}}=\left[
\begin{array}{c|c|c|c|c}
0 & 0 & {\mathrm{O}}_{1\times t} & 1 & 1 \\
\hline
0 & 0 & {\mathrm{O}}_{1 \times t} & 0 & 1 \\
\hline
{\mathrm{O}}_{t \times 1} & {\mathrm{O}}_{t \times 1} & {\mathrm{I}} & {\mathrm{N}}_{t \times 1} & {\mathrm{O}}_{t \times 1} \\
\hline
0 & 1 & {\mathrm{O}}_{1 \times t} & 0 & 0 \\
\hline
0 & 0 & {\mathrm{N}}_{1 \times t} & 0 & 0 \\
\end{array}
\right],\; {\mathrm{B}}= \left[
\begin{array}{c}
{\mathrm{O}}_{1\times t} \\
\hline
{\mathrm{N}}_{1\times t} \\
\hline
{\mathrm{T}}_{r,0} \\
\hline
{\mathrm{O}}_{1\times t} \\
\hline
{\mathrm{O}}_{1\times t} \\
\end{array}
\right],\; {\mathrm{C}}=\left[
\begin{array}{c|c|c|c|c}
{\mathrm{N}}_{t \times 1} & {\mathrm{O}}_{t \times 1} & {\mathrm{I}}+{\mathrm{T}}_{r,0} & {\mathrm{O}}_{t\times 1} & {\mathrm{O}}_{t \times 1}\\
\end{array}
\right],\; {\mathrm{D}}={\mathrm{I}} ,
\end{equation*}
}}
we see that ${\mathrm{M}}_o$ is invertible over $\mathbb{F}_2$ if and only if the $(t+4) \times (t+4)$
{\scriptsize{
\begin{equation*}
 {\mathrm{M}}_o'=\left[
\begin{array}{c|c|c|c|c}
0 & 0 & {\mathrm{O}}_{1\times t} & 1 & 1 \\
\hline
1 & 0 & {\mathrm{N}}_{1 \times t} & 0 & 1  \\
\hline
{\mathrm{O}}_{t \times 1} & {\mathrm{O}}_{t \times 1} & {\mathrm{I}} & {\mathrm{N}}_{t \times 1} & {\mathrm{O}}_{t \times 1} \\
\hline
0 & 1 & {\mathrm{O}}_{1 \times t} & 0 & 0 \\
\hline
0 & 0 & {\mathrm{N}}_{1 \times t} & 0 & 0 \\
\end{array}
\right]. 
\end{equation*}
}}
is invertible over $\mathbb{F}_2$. Consider the permutation group $\mathrm{S}_k$ on $k$ distinct symbols, say $\{1,2,\ldots,k\}$. 
A permutation $\sigma \in \mathrm{S}_k$ acts on the columns $C_i$ (respectively, rows $R_i$) by
\[
\sigma \cdot C_i := C_{\sigma(i)} \quad \text{(respectively, } \sigma \cdot R_i := R_{\sigma(i)}\text{)}.
\]
Since $\mathrm{M}'_o$ is a square matrix of size $t+4$, we take $k = t+4$. 
Consider the permutation
\[
\sigma =
\begin{pmatrix}
1 & 2 & 3 & 4 & 5 & 6 & 7 & \cdots & t+3 & t+4 \\
1 & 2 & t+4 & t+3 & 3 & 4 & 5 & \cdots & t+1 & t+2
\end{pmatrix} \in S_{t+4}.
\]
Applying $\sigma$ to the columns of $\mathrm{M}'_o$, and subsequently to the rows, we obtain

%and get a matrix, say $\mathrm{M}_o'(C)$. Now we apply the same permutation $\sigma=(3 \,\, t+4 \,\, t+2  \, \ldots \, 5)(4 \,\, t+3 \,\, t+1  \, \ldots \, 6)$ on the rows of $\mathrm{M}_o'(C)$.} %Consider the permutation group $\mathrm{S}_n$ with $n$ distinct symbols, say $\{1,2,\ldots,n\}$. Denote the columns (respectively rows) of a $(t+4) \times (t+4)$ matrix by $C_1, C_2, \ldots, C_{t+4}$ (respectively $R_1, R_2, \ldots, R_{t+4}$). For $n=t+4$, we identify the symbols $\{1,2,\ldots,t+4\}$ with $\{C_1, C_2, \ldots, C_{t+4}\}$. Then we apply the permutation $(3 \,\, t+4 \,\, t+2  \, \ldots \, 5)(4 \,\, t+3 \,\, t+1  \, \ldots \, 6)$ on the columns of $\mathrm{M}_o'$. Next we identify $\{1,2,\ldots,t+4\}$ with $\{R_1, R_2, \ldots, R_{t+4}\}$ and apply the permutation $(3 \,\, t+4 \,\, t+2  \, \ldots \, 5)$ the rows. {\color{red}Now we first make the following column operations : replace $C_3$ by $C_{t+4}$, $C_4$ by $C_{t+3}$, and for all $5 \leq i \leq t+4$, replace the columns $C_i$ by $C_{i-2}$. Next we make the following row operations}:  replace $R_3$ by $R_{t+4}$, $R_4$ by $R_{t+3}$, then for all $5 \leq i \leq t+4$, we replace $R_i$ by $R_{i-2}$.
%By doing all these column and row operations, we get the following matrix :  %$C_{i} \to C_{i+t}$ for all $3 \leq i \leq t+2$, $C_{t+3} \to C_4, C_{t+4} \to C_3$, and then the following row operations:  $R_{i} \to R_{i+2}$ for all $3 \leq i \leq t+2$, $R_{t+3} \to R_4, R_{t+4} \to R_3$. %columns $C_3$ with $C_{t+4}$, $C_{4}$ with $C_{t+3}$, then swap the rows $R_3$ with $R_{t+4}$, $R_{4}$ with $R_{t+3}$. 
{\scriptsize{
\begin{equation*}
 {\mathrm{M}}_o^{\prime \prime}=\left[
\begin{array}{c|c|c|c|c}
0 & 0 & 1 & 1 & {\mathrm{O}}_{1\times t} \\
\hline
1 & 0 & 1 & 0 & {\mathrm{N}}_{1 \times t}  \\
\hline
0 & 0 &  0 & 0 & {\mathrm{N}}_{1 \times t} \\
\hline
0 & 1 & 0 & 0 & {\mathrm{O}}_{1 \times t} \\
\hline
{\mathrm{O}}_{t \times 1} & {\mathrm{O}}_{t \times 1} & {\mathrm{O}}_{t \times 1} & {\mathrm{N}}_{t \times 1} & {\mathrm{I}} \\
\end{array}
\right]. 
\end{equation*}
}}
Note that, rank of a matrix is invariant under permutations of rows and columns. Finally taking,
{\scriptsize{
\begin{equation*}
{\mathrm{A}}=\left[
\begin{array}{c|c|c|c}
0 & 0 & 1 & 1 \\
\hline
1 & 0 & 1 & 0  \\
\hline
0 & 0 & 0 & 0 \\
\hline
0 & 1 & 0 & 0 \\
\end{array}
\right], \quad {\mathrm{B}}= \left[
\begin{array}{c}
{\mathrm{O}}_{1\times t} \\
\hline
{\mathrm{N}}_{1\times t} \\
\hline
{\mathrm{N}}_{1\times t} \\
\hline
{\mathrm{O}}_{1\times t} \\
\end{array}
\right], \quad {\mathrm{C}}=\left[
\begin{array}{c|c|c|c}
{\mathrm{O}}_{t \times 1} & {\mathrm{O}}_{t \times 1} & {\mathrm{O}}_{t\times 1} & {\mathrm{N}}_{t \times 1} \\
\end{array}
\right], \quad {\mathrm{D}}={\mathrm{I}}
\end{equation*}
}}
and applying Lemma \ref{schur}, we see that ${\mathrm{M}}_o''$ is invertible over $\mathbb{F}_2$ since the matrix {\scriptsize{$\left[
\begin{array}{c|c|c|c}
0 & 0 & 1 & 1 \\
\hline
1 & 0 & 1 & 1  \\
\hline
0 & 0 & 0 & 1 \\
\hline
0 & 1 & 0 & 0 \\
\end{array}
\right]$}}
is invertible over $\mathbb{F}_2$. This shows that the Monsky matrix ${\mathrm{M}}_o$ is invertible over $\mathbb{F}_2$ since $\operatorname{rank}_{\F_2}{(\mathrm{M}_{o})} = \operatorname{rank}_{\F_2}{(\mathrm{M}'_{o})} = \operatorname{rank}_{\F_2}{(\mathrm{M}''_{o})}$. The case $\alpha =-1$, can be handled in similar fashion.
\end{proof}

\begin{proof}[Proof of Theorem \ref{533}]
 Let \( n=2\prod_{i=1}^{t} p_i q_i r_i \) be as in Theorem~\ref{533}. For simpler approach in the proof, we regard it as
$ n=2\left(\prod_{i=1}^{t} p_i\right)\left(\prod_{i=1}^{t} q_i\right)\left(\prod_{i=1}^{t} r_i\right) $. We present the proofs of Case~A and Case~B separately. For clarity, given a matrix \(\mathrm{Z}\), we denote by \(\mathrm{Z}(A)\) and \(\mathrm{Z}(B)\) its forms in Case~A and Case~B, respectively, whenever necessary. Note that \( \mathrm{D}_{rq} \equiv_{2} \mathrm{D}_{qr} + \mathrm{I} \), since \( \legendre{q_i}{r_i} = - \legendre{r_i}{q_i} \) for all \(i\). Furthermore, from \eqref{matrix D and E}, we have that, 
{\scriptsize{
 \begin{equation}
 \label{D533} 
{\mathrm{D}}_2=\left[
\begin{array}{c|c|c}
{\mathrm{I}} & {\mathrm{O}} & {\mathrm{O}} \\
\hline
{\mathrm{O}} & {\mathrm{I}} & {\mathrm{O}} \\
\hline
{\mathrm{O}} & {\mathrm{O}} & {\mathrm{I}} \\
\end{array}
\right],\quad {\mathrm{D}}_{-1}=\left[
\begin{array}{c|c|c}
{\mathrm{O}} & {\mathrm{O}} & {\mathrm{O}} \\
\hline
{\mathrm{O}} & {\mathrm{I}} & {\mathrm{O}} \\
\hline
{\mathrm{O}} & {\mathrm{O}} & {\mathrm{I}} \\
\end{array}
\right].     
\end{equation} 
}}

\noindent \textbf{Case A :}  Assume that  $\alpha=1$.  Therefore we have,
{\scriptsize{ 
 \begin{equation*}
{\mathrm{E}}(A)={\mathrm{E}} \equiv_{2} \left[
\begin{array}{c|c|c}
 {\mathrm{O}} & {\mathrm{O}} & {\mathrm{O}} \\
\hline
{\mathrm{O}} & {\mathrm{T}}_{q, t+1}+{\mathrm{D}}_{qr} & {\mathrm{N}}+{\mathrm{I}}+{\mathrm{D}}_{qr} \\
\hline
{\mathrm{O}} & {\mathrm{D}}_{qr}+{\mathrm{I}} & {\mathrm{T}}_{r,1}+{\mathrm{D}}_{qr} \\
\end{array}
\right].     
\end{equation*} 
}}
%\begin{equation}\label{E533}% \end{equation}   
By the definition of Monsky matrix as in \eqref{monskymat} we have, 
{\scriptsize{
\begin{equation*}
 {\mathrm{M}}_e(A)= {\mathrm{M}}_e \equiv_{2} \left[
\begin{array}{c|c|c|c|c|c}
{\mathrm{I}} & {\mathrm{O}} & {\mathrm{O}} & {\mathrm{I}} & {\mathrm{O}} & {\mathrm{O}} \\
\hline
{\mathrm{O}} & {\mathrm{I}} & {\mathrm{O}} & {\mathrm{O}} & {\mathrm{T}}_{q,t}+{\mathrm{D}}_{qr} & {\mathrm{N}}+{\mathrm{I}}+{\mathrm{D}}_{qr} \\
\hline
{\mathrm{O}} & {\mathrm{O}} & {\mathrm{I}} & {\mathrm{O}} & {\mathrm{D}}_{qr}+{\mathrm{I}} & {\mathrm{T}}_{r,0}+{\mathrm{D}}_{qr} \\
\hline
{\mathrm{I}} & {\mathrm{O}} & {\mathrm{O}} & {\mathrm{O}} & {\mathrm{O}} & {\mathrm{O}} \\
\hline
{\mathrm{O}} & {\mathrm{T}}_{q,t}^{\top}+{\mathrm{D}}_{qr} & {\mathrm{D}}_{qr}+{\mathrm{I}} & {\mathrm{O}} & {\mathrm{I}} & {\mathrm{O}} \\
\hline
{\mathrm{O}} & {\mathrm{N}}+{\mathrm{I}}+{\mathrm{D}}_{qr} & {\mathrm{T}}_{r,0}^{\top}+{\mathrm{D}}_{qr} & {\mathrm{O}} & {\mathrm{O}} & {\mathrm{I}} \\
\end{array}
\right]. 
\end{equation*}
}}
By Lemma~\ref{schur} (taking \(\mathrm{A}=\mathrm{D}_2\)) together with Lemma~\ref{itemize}, it follows that \(\mathrm{M}_e(A)\) is invertible over \(\mathbb{F}_2\) if and only if 
{\scriptsize{
\[
\mathrm{M}_e'(A)=
\left[\begin{array}{c|c|c}
\mathrm{I} & \mathrm{O} & \mathrm{O} \\
\hline
\mathrm{O} & \mathrm{S}_1 & \mathrm{S}_2 \\
\hline
\mathrm{O} & \mathrm{S}_3 & \mathrm{S}_4
\end{array}\right]
\]
}}
is invertible over \(\mathbb{F}_2\), where 
\begin{align*}
      {\mathrm{S}}_1 &\equiv_{2} {\mathrm{N}}+{\mathrm{T}}_{q,0}^{\top}{\mathrm{D}}_{qr}+{\mathrm{D}}_{qr}{\mathrm{T}}_{q,0}  \\ %$+{\mathrm{I}}$
     {\mathrm{S}}_2 &\equiv_{2} {\mathrm{N}}+{\mathrm{D}}_{qr}{\mathrm{N}} + {\mathrm{T}}_{q,t}^{\top}+{\mathrm{T}}_{q,t}^{\top}{\mathrm{D}}_{qr} + {\mathrm{D}}_{qr}{\mathrm{T}}_{r,0} + {\mathrm{T}}_{r,0},\\ 
     {\mathrm{S}}_3 &\equiv_{2} {\mathrm{S}}_2^{\top},\\
    {\mathrm{S}}_4 &\equiv_{2}{\mathrm{N}}+{\mathrm{N}}{\mathrm{D}}_{qr}+ {\mathrm{D}}_{qr}{\mathrm{N}}+{\mathrm{T}}_{r,0}^{\top}{\mathrm{D}}_{qr}+{\mathrm{D}}_{qr}{\mathrm{T}}_{r,0}. %$+{\mathrm{I}}$. 
\end{align*}
\noindent(i) Assume that $\mu=1$. Then we have ${\mathrm{D}}_{qr}={\mathrm{I}}$. Hence ${\mathrm{S}}_1 \equiv_{2} {\mathrm{S}}_4 \equiv_{2} {\mathrm{I}}$,\;\;\; ${\mathrm{S}}_2 \equiv_{2} \mathrm{O}$ by Lemma \ref{itemize}. In this case, 
    ${\mathrm{M}}_e' (A) \equiv_{2}$ {\scriptsize{ 
    $\left[\begin{array}{c|c|c} 
    {\mathrm{I}} & {\mathrm{O}} & {\mathrm{O}} \\ \hline 
    {\mathrm{O}} & {\mathrm{I}} & {\mathrm{O}} \\ \hline 
    {\mathrm{O}} & {\mathrm{O}} & {\mathrm{I}} \\ 
    \end{array}\right]$}} 
 which is invertible over $\F_2$.
    
\noindent (ii)  Assume that $\mu=-1$, which implies that $\mathrm{D}_{qr}=\mathrm{O}$. Moreover, the condition $\mu_1=\mu_2$ yields $\mathrm{T}_{q,0}=\mathrm{T}_{r,0}$. Consequently, we obtain $\mathrm{S}_1 \equiv_{2} \mathrm{S}_4 \equiv_{2} \mathrm{N}$ and $\mathrm{S}_2 \equiv_{2} (t+1)\mathrm{I}$ by Lemma \ref{itemize}. If $t=1$, then ${\mathrm{M}}_e'(A) \equiv_{2}$  {\scriptsize{ $\left[\begin{array}{c|c|c}
1 & 0 & 0 \\
\hline
0 & 1 & 0 \\
\hline
0 & 0 & 1 \\
\end{array}\right]$}}, which is invertible. If $t$ is even  then  ${\mathrm{M}}_e'(A)$ is equivalent to  {\scriptsize{$\left[\begin{array}{c|c|c}
{\mathrm{I}} & {\mathrm{O}} & {\mathrm{O}} \\
\hline
{\mathrm{O}} & {\mathrm{N}} & {\mathrm{I}} \\
\hline
{\mathrm{O}} & {\mathrm{I}} & {\mathrm{N}} \\
\end{array}\right]$}} modulo $2$. By Lemma \ref{determinant}, it is invertible over $\mathbb{F}_{2}$  since {\scriptsize{$\left[\begin{array}{c|c}
 {\mathrm{N}} & {\mathrm{I}} \\
\hline
 {\mathrm{I}} & {\mathrm{N}} \\
\end{array}\right]$}} is invertible over $\mathbb{F}_{2}$, (this is because ${\mathrm{N}}^2 \equiv_{2} t\mathrm{N} \equiv_{2} \mathrm{O}$ by Lemma \ref{itemize}, together with suitable elementary row operations).

 Now consider the case $ \mu_1 \neq \mu_2 $. Without loss of generality, assume that $ \mathrm{T}_{q,0}=\mathrm{U}_{0} $ and $ \mathrm{T}_{r,0}=\mathrm{L}_{0} $. Hence,
$ \mathrm{S}_1 \equiv_{2} \mathrm{S}_4 \equiv_{2} \mathrm{N} $, and
$ \mathrm{S}_2 \equiv_{2} \mathrm{N} + \mathrm{T}_{q,t}^{\top} + \mathrm{T}_{r,0} $. In fact,
$ \mathrm{T}_{q,t}^{\top} + \mathrm{T}_{r,0}
= \mathrm{U}_{0}^{\top} + \mathrm{L}_{0} + t\mathrm{I} \equiv_{2} \mathrm{I}$
by part (c) of Lemma \ref{UL-LU}. In this case, ${\mathrm{M}}_e'(A) \equiv $ {\scriptsize{$ \left[\begin{array}{c|c|c}
{\mathrm{I}} & {\mathrm{O}} & {\mathrm{O}} \\
\hline
{\mathrm{O}} & {\mathrm{N}} & {\mathrm{N}}+{\mathrm{I}} \\
\hline
{\mathrm{O}} & {\mathrm{N}}+{\mathrm{I}} & {\mathrm{N}} \\
\end{array}\right]$}} which is invertible over $\F_2$ (by Lemma \ref{schur}) if and only if {\scriptsize{$\left[\begin{array}{c|c}
{\mathrm{N}} & {\mathrm{N}}+{\mathrm{I}} \\
\hline
{\mathrm{N}}+{\mathrm{I}} & {\mathrm{N}} \\
\end{array}\right]$}} is invertible over  $\F_2$.  A basic column operation shows that the rank of
{\scriptsize{$\left[\begin{array}{c|c}
{\mathrm{N}} & {\mathrm{N}}+{\mathrm{I}} \\
\hline
{\mathrm{N}}+{\mathrm{I}} & {\mathrm{N}} \\
\end{array}\right]$}}
coincides with the rank of
{\scriptsize{$\left[\begin{array}{c|c}
{\mathrm{I}} & {\mathrm{N}}+{\mathrm{I}} \\
\hline
{\mathrm{I}} & {\mathrm{N}} \\
\end{array}\right]$}}
over $\mathbb{F}_2$. Moreover, by Lemma~\ref{schur}, the latter matrix is invertible over $\mathbb{F}_2$. This completes the proof for Case A. \\

\noindent \textbf{Case B:} Assume that $\alpha=-1$. Therefore, we have that 
{\scriptsize{
 \begin{equation*}
 %\label{D533} 
 {\mathrm{E}}(B) \equiv_{2} \left[
\begin{array}{c|c|c}
 {\mathrm{O}} & {\mathrm{I}} & {\mathrm{I}} \\
\hline
{\mathrm{I}} & {\mathrm{T}}_{q, t}+{\mathrm{D}}_{qr} & {\mathrm{N}}+{\mathrm{I}}+{\mathrm{D}}_{qr} \\
\hline
{\mathrm{I}} & {\mathrm{D}}_{qr}+{\mathrm{I}} & {\mathrm{T}}_{r,0}+{\mathrm{D}}_{qr} 
\end{array}
\right], \quad  \quad {\mathrm{M}}_e(B)=\begin{bmatrix}
\begin{array}{c c}
{\mathrm{D}}_2 & {\mathrm{E}}(B)+{\mathrm{D}}_2 \\
{\mathrm{E}}(B)^{\top}+{\mathrm{D}}_{2} & {\mathrm{D}}_{-1}
\end{array}
\end{bmatrix},  
\end{equation*} 
}}
where ${\mathrm{D}}_2, \; {\mathrm{D}}_{-1}$ are as in \eqref{D533}.
Therefore, by taking $\mathrm{D}=\mathrm{D}_{-1}$ and applying Lemma~\ref{schur}, we conclude that
\begin{align*}
 \det {\mathrm{M}}_e(B) & = \det \left({\mathrm{D}}_{-1}+({\mathrm{E}}(B)^{\top}+{\mathrm{D}}_{2})({\mathrm{E}}(B)+{\mathrm{D}}_2)\right) \\ 
   &=\det \left({\mathrm{D}}_{-1}+{\mathrm{E}}(B)^{\top}\mathrm{E}(B)+ {\mathrm{E}}(B)^{\top}+\mathrm{E}(B)+{\mathrm{D}}_2\right).
\end{align*}
Let ${\mathrm{M}}^{\prime}_e(B)=\left({\mathrm{D}}_{-1}+{\mathrm{E}}(B)^{\top}\mathrm{E}(B)+ {\mathrm{E}}(B)^{\top}+\mathrm{E}(B)+{\mathrm{D}}_2\right)$, Therefore, 
{\scriptsize{
$${\mathrm{M}}^{\prime}_e (B) \equiv_{2} \left[\begin{array}{c|c|c}
{\mathrm{I}} & {\mathrm{T}}_{q, t+1} &  {\mathrm{N}}+{\mathrm{I}}+{\mathrm{T}}_{r, 0} \\
\hline
{\mathrm{T}}_{q, t+1}^{\top} & {\mathrm{R}}_1 & {\mathrm{R}}_2 \\
\hline
{\mathrm{N}}+{\mathrm{I}}+{\mathrm{T}}_{r, 0}^{\top}  & {\mathrm{R}}_3 & {\mathrm{R}}_4 \\
\end{array}\right]$$ 
}}
where 
\begin{align*}
      {\mathrm{R}}_1 &\equiv_{2} {\mathrm{I}}+{\mathrm{T}}_{q,0}^{\top}{\mathrm{D}}_{qr}+{\mathrm{D}}_{qr}{\mathrm{T}}_{q,0},\\ %$+{\mathrm{I}}$
     {\mathrm{R}}_2 &\equiv_{2} {\mathrm{I}}+{\mathrm{N}}+{\mathrm{D}}_{qr}{\mathrm{N}} + {\mathrm{T}}_{q,t}^{\top}+ {\mathrm{T}}_{q,t}^{\top}\mathrm{N}+{\mathrm{T}}_{q,t}^{\top}{\mathrm{D}}_{qr} + {\mathrm{D}}_{qr}{\mathrm{T}}_{r,0} + {\mathrm{T}}_{r,0},\\ 
     {\mathrm{R}}_3 &\equiv_{2} {\mathrm{R}}_2^{\top},\\
    {\mathrm{R}}_4 &\equiv_{2}{\mathrm{I}}+{\mathrm{N}}{\mathrm{D}}_{qr}+ {\mathrm{D}}_{qr}{\mathrm{N}}+{\mathrm{T}}_{r,0}^{\top}{\mathrm{D}}_{qr}+{\mathrm{D}}_{qr}{\mathrm{T}}_{r,0}. %$+{\mathrm{I}}$. 
\end{align*}
\noindent(i) Assume that $\mu=-1$, which implies that $\mathrm{D}_{qr}=\mathrm{O}$. Consequently, we obtain
$\mathrm{R}_1 \equiv_{2} \mathrm{R}_4 \equiv_{2} \mathrm{I}$,\;\;\; $\mathrm{R}_2 \equiv_{2} \mathrm{I} + \mathrm{T}_{q,t}^{\top} + \mathrm{T}_{r,0}$. Again, applying Lemma~\ref{schur} with $\mathrm{A}=\mathrm{I}$, we deduce that $\mathrm{M}'_e(B)$ is invertible over $\mathbb{F}_2$ if and only if  ${\mathrm{M}}^{\prime \prime}_e(B)=$ {\scriptsize{ $\left[\begin{array}{c|c}
{\mathrm{I}} & {\mathrm{T}_{r, 0}}+ \mathrm{T_{q, t+1}^{\top}} {\mathrm{T}_{r, 0}}\\
\hline
\mathrm{T}_{r, 0}^{\top} + \mathrm{T}_{r, 0}^{\top} \mathrm{T_{q, t+1}} & {\mathrm{I}}
\end{array}\right]$}} is invertible over $\F_2$. Again, applying Lemma \ref{schur} on ${\mathrm{M}}^{\prime \prime}_e(B)$ (by considering $\mathrm{D}=\mathrm{I}$), we get that $\det {\mathrm{M}}^{\prime \prime}_e(B) = \det \left(\mathrm{I}+ ({\mathrm{T}_{r, 0}}+ \mathrm{T_{q, t+1}^{\top}} {\mathrm{T}_{r, 0}})(\mathrm{T}_{r, 0}^{\top} + \mathrm{T}_{r, 0}^{\top} \mathrm{T_{q, t+1}})\right)$. By Lemma \ref{itemize}, it can be seen that $({\mathrm{T}_{r, 0}}+ \mathrm{T_{q, t+1}^{\top}} {\mathrm{T}_{r, 0}})(\mathrm{T}_{r, 0}^{\top} + \mathrm{T}_{r, 0}^{\top} \mathrm{T_{q, t+1}}) \equiv_{2} \mathrm{O}$. Hence $\det {\mathrm{M}}^{\prime \prime}_e (B) = \det {\mathrm{M}}^{\prime}_e (B)= \det {\mathrm{M}}_e (B) \equiv_2 1$.

\vspace{0.5cm}

\noindent(ii) Assume that $\mu =1$, i.e., ${\mathrm{D}}_{qr}={\mathrm{I}}$ and that gives us ${\mathrm{R}}_1 \equiv_{2} {\mathrm{R}}_4 \equiv_{2} \mathrm{N}$,\;\; $\;{\mathrm{R}}_2 \equiv_{2} {\mathrm{I}}+ {\mathrm{N}}$ by Lemma \ref{itemize}.
Again applying Lemma \ref{schur} (by taking $\mathrm{A}=\mathrm{I}$), we can see that ${\mathrm{M}}^{\prime}_e (B)$ is invertible over $\F_{2}$ if and only if  ${\mathrm{M}}^{\prime \prime}_e(B)=$   {\scriptsize{$\left[\begin{array}{c|c}
{\mathrm{N}} &  {\mathrm{I}}+{\mathrm{N}}+\mathrm{T_{q, t+1}^{\top}}+ \mathrm{T_{q, t+1}^{\top}} {\mathrm{T}_{r, 0}}\\
\hline
 {\mathrm{I}}+{\mathrm{N}}+\mathrm{T_{q, t+1}}+ {\mathrm{T}}_{r, 0}^{\top} \mathrm{T_{q, t+1}}  & {\mathrm{N}}
\end{array}\right]$}} is invertible over $\F_2$.

Now assume that $\mu_{1}=\mu_{2}$, that is $\mathrm{T}_{q,0}=\mathrm{T}_{r, 0}$. And by Lemma \ref{itemize}, 
\begin{align*}
\mathrm{T_{q, t+1}^{\top}}+ \mathrm{T_{q, t+1}^{\top}} {\mathrm{T}_{r, 0}}& = \mathrm{T_{q, 0}^{\top}} {\mathrm{T}_{r, 0}}+ (t+1){\mathrm{T}_{r, 0}}+ (t+1)\mathrm{I}+ \mathrm{T_{q,0}^{\top}} \\ &\equiv_{2} t(\mathrm{N}+\mathrm{T}_{r, 0}+\mathrm{I})  \quad \text{(By Lemma \ref{itemize})}. 
\end{align*}
Therefore, if $t$ is even then ${\mathrm{M}}^{\prime \prime}_e (B)$  is congruent to {\scriptsize{$ \left[\begin{array}{c|c}
{\mathrm{N}} &  {\mathrm{I}}+{\mathrm{N}}\\
\hline
 {\mathrm{I}}+{\mathrm{N}} & {\mathrm{N}}
\end{array}\right]$}} modulo $2$ and it is invertible by Lemma~\ref{schur}, together with suitable elementary row and column operations. Also if $t=1$, then ${\mathrm{M}}^{\prime \prime}_e (B)=$ {\scriptsize{ $\left[\begin{array}{c|c}
1 &  0\\
\hline
 0 & 1
\end{array}\right]$}}, which is invertible.

Consider the case $\mu_{1} \neq \mu_{2}$, then without loss of generality assume that $\mathrm{T}_{q, 0}= \mathrm{U}_{0}$ and $\mathrm{T}_{r, 0}= \mathrm{L}_{0}$. Then 
\begin{align*}
   \mathrm{T_{q, t+1}^{\top}}+ \mathrm{T_{q, t+1}^{\top}} {\mathrm{T}_{r, 0}}& = \mathrm{U}_{0}^{\top} \mathrm{L}_{0}+ \mathrm{U}_{0}^{\top}+ (t+1)\mathrm{L}_{0}+ (t+1)\mathrm{I}= \mathrm{W},  \;\; \text{(say)}  
\end{align*}
and in fact, 
\begin{align*}
\mathrm{W}(i,j) &=
        \begin{cases}
            0  & \text{if } i>j,\\
            (t-i)(i-1)+ (t-i)+(t+1)(i-1)+(t+1)= i(2t-i+1) \equiv_{2} 0  & \text{if } i=j,\\
            t-2+(t+1)+1=2t \equiv_{2} 0  & \text{if } i<j,
        \end{cases}   
\end{align*}
by Lemma \ref{UL-LU}. Therefore, we have that ${\mathrm{M}}^{\prime \prime}_e (B) \equiv_{2}$  {\scriptsize{$\left[\begin{array}{c|c}
{\mathrm{N}} &  {\mathrm{I}}+{\mathrm{N}}\\
\hline
 {\mathrm{I}}+{\mathrm{N}} & {\mathrm{N}}
\end{array}\right]$}}  and it is invertible by Lemma~\ref{schur}, together with suitable elementary row and column operations. In the other case, namely when $ \mathrm{T}_{q,0}= \mathrm{L}_{0} $ and $ \mathrm{T}_{r,0}= \mathrm{U}_{0} $, a similar argument, combined with Lemma~\ref{UL-LU}, shows that $ \mathrm{M}''_{e}(B)$ is invertible. This completes the proof.
\end{proof}

\begin{proof}[Proof of Theorem \ref{1357}]
 For simplicity in the proof, we regard  $n$ as
$$ n=\left(\prod_{i=1}^{t} p_i\right)\left(\prod_{i=1}^{t} s_i\right)\left(\prod_{i=1}^{t} q_i\right)\left(\prod_{i=1}^{t} r_i\right). $$ In this case from \eqref{matrix D and E}, we have 
{\scriptsize{
\begin{equation*}
 %\label{D1357} 
{\mathrm{D}}_2=\left[
\begin{array}{c|c|c|c}
{\mathrm{O}} & {\mathrm{O}} & {\mathrm{O}} & {\mathrm{O}} \\
\hline
{\mathrm{O}} & {\mathrm{O}} & {\mathrm{O}} & {\mathrm{O}} \\
\hline
{\mathrm{O}} & {\mathrm{O}} & {\mathrm{I}} & {\mathrm{O}} \\
\hline 
{\mathrm{O}} & {\mathrm{O}} & {\mathrm{O}} & {\mathrm{I}} \\
\end{array}
\right],\quad {\mathrm{D}}_{-2}=\left[
\begin{array}{c|c|c|c}
{\mathrm{O}} & {\mathrm{O}} & {\mathrm{O}} & {\mathrm{O}} \\
\hline
{\mathrm{O}} & {\mathrm{I}} & {\mathrm{O}} & {\mathrm{O}} \\
\hline
{\mathrm{O}} & {\mathrm{O}} & {\mathrm{O}} & {\mathrm{O}} \\
\hline 
{\mathrm{O}} & {\mathrm{O}} & {\mathrm{O}} & {\mathrm{I}} \\
\end{array}
\right], \quad {\mathrm{E}}=\left[
\begin{array}{c|c|c|c}
 {\mathrm{I}} & {\mathrm{O}} & {\mathrm{I}} & {\mathrm{O}}\\
\hline
{\mathrm{O}} & {\mathrm{T}}_{s,1} & {\mathrm{O}} & {\mathrm{I}} \\
\hline
{\mathrm{I}} & {\mathrm{N}} & {\mathrm{T}}_{q,t+1} & {\mathrm{O}} \\
\hline
{\mathrm{O}} & {\mathrm{I}} & {\mathrm{O}} & {\mathrm{I}} \\
\end{array}
\right].    
\end{equation*}   
}}
 By the definition of Monsky matrix as in \eqref{monskymat}, we have  
{\scriptsize{
 \begin{equation*}
 {\mathrm{M}}_o\equiv_2\left[
\begin{array}{c|c|c|c|c|c|c|c}
{\mathrm{O}} & {\mathrm{O}} & {\mathrm{O}} & {\mathrm{O}} & {\mathrm{I}} & {\mathrm{O}} & {\mathrm{I}} & {\mathrm{O}}\\
\hline
{\mathrm{O}} & {\mathrm{O}} & {\mathrm{O}} & {\mathrm{O}} & {\mathrm{O}} & {\mathrm{T}}_{s,1} & {\mathrm{O}} & {\mathrm{I}} \\
\hline
{\mathrm{O}} & {\mathrm{O}} & {\mathrm{I}} & {\mathrm{O}} & {\mathrm{I}} & {\mathrm{N}} & {\mathrm{T}}_{q,t} & {\mathrm{O}}\\
\hline
{\mathrm{O}} & {\mathrm{O}} & {\mathrm{O}} & {\mathrm{I}} & {\mathrm{O}} & {\mathrm{I}} & {\mathrm{O}} & {\mathrm{O}} \\
\hline
{\mathrm{I}} & {\mathrm{O}} & {\mathrm{I}} & {\mathrm{O}} &{\mathrm{O}} & {\mathrm{O}} & {\mathrm{O}} & {\mathrm{O}} \\
\hline
{\mathrm{O}} & {\mathrm{T}}_{s,0} & {\mathrm{O}} & {\mathrm{I}} & {\mathrm{O}} & {\mathrm{O}} & {\mathrm{O}} & {\mathrm{O}}\\
\hline
{\mathrm{I}} & {\mathrm{N}} & {\mathrm{T}}_{q,t+1} & {\mathrm{O}} & {\mathrm{O}} & {\mathrm{O}} & {\mathrm{I}} & {\mathrm{O}} \\
\hline 
{\mathrm{O}} & {\mathrm{I}} & {\mathrm{O}} & {\mathrm{O}} & {\mathrm{O}} & {\mathrm{O}} & {\mathrm{O}} & {\mathrm{I}} \\
\end{array}
\right] \xrightarrow[\text{with ${\mathrm{R}}_{5+i}$ for $1 \leq i \leq t$}]{\text{swap the rows ${\mathrm{R}}_{3+i}$}}  \left[
\begin{array}{c|c|c|c|c|c|c|c}
{\mathrm{O}} & {\mathrm{O}} & {\mathrm{O}} & {\mathrm{O}} & {\mathrm{I}} & {\mathrm{O}} & {\mathrm{I}} & {\mathrm{O}}\\
\hline
{\mathrm{O}} & {\mathrm{O}} & {\mathrm{O}} & {\mathrm{O}} & {\mathrm{O}} & {\mathrm{T}}_{s,1} & {\mathrm{O}} & {\mathrm{I}} \\
\hline
{\mathrm{O}} & {\mathrm{O}} & {\mathrm{I}} & {\mathrm{O}} & {\mathrm{I}} & {\mathrm{N}} & {\mathrm{T}}_{q,t} & {\mathrm{O}}\\
\hline
{\mathrm{O}} & {\mathrm{T}}_{s,0} & {\mathrm{O}} & {\mathrm{I}} & {\mathrm{O}} & {\mathrm{O}} & {\mathrm{O}} & {\mathrm{O}}\\
\hline
{\mathrm{I}} & {\mathrm{O}} & {\mathrm{I}} & {\mathrm{O}} &{\mathrm{O}} & {\mathrm{O}} & {\mathrm{O}} & {\mathrm{O}} \\
\hline

{\mathrm{O}} & {\mathrm{O}} & {\mathrm{O}} & {\mathrm{I}} & {\mathrm{O}} & {\mathrm{I}} & {\mathrm{O}} & {\mathrm{O}} \\
\hline
{\mathrm{I}} & {\mathrm{N}} & {\mathrm{T}}_{q,t+1} & {\mathrm{O}} & {\mathrm{O}} & {\mathrm{O}} & {\mathrm{I}} & {\mathrm{O}} \\
\hline 
{\mathrm{O}} & {\mathrm{I}} & {\mathrm{O}} & {\mathrm{O}} & {\mathrm{O}} & {\mathrm{O}} & {\mathrm{O}} & {\mathrm{I}} \\
\end{array}
\right]. 
\end{equation*}
}}
By considering 
{\scriptsize{

\begin{equation*} 
{\mathrm{A}}=\left[
\begin{array}{c|c|c|c|c}
{\mathrm{O}} & {\mathrm{O}} & {\mathrm{O}} & {\mathrm{O}} & {\mathrm{I}} \\
\hline
{\mathrm{O}} & {\mathrm{O}} & {\mathrm{O}} & {\mathrm{O}} & {\mathrm{O}}  \\
\hline
{\mathrm{O}} & {\mathrm{O}} & {\mathrm{I}} & {\mathrm{O}} & {\mathrm{I}} \\
\hline
{\mathrm{O}} & {\mathrm{T}}_{s,0} & {\mathrm{O}} & {\mathrm{I}} & {\mathrm{O}} \\
\hline
{\mathrm{I}} & {\mathrm{O}} & {\mathrm{I}} & {\mathrm{O}} &{\mathrm{O}} \\
\end{array}
\right], {\mathrm{B}}=\left[
\begin{array}{c|c|c}
{\mathrm{O}} & {\mathrm{I}} & {\mathrm{O}} \\
\hline
{\mathrm{T}}_{s,1} & {\mathrm{O}} & {\mathrm{I}}\\
\hline
{\mathrm{N}} & {\mathrm{T}}_{q,t} & {\mathrm{O}}\\
\hline
{\mathrm{O}} & {\mathrm{O}} & {\mathrm{O}} \\
\hline
{\mathrm{O}} & {\mathrm{O}} & {\mathrm{O}}
\end{array}
\right],\quad {\mathrm{C}}=\left[
\begin{array}{c|c|c|c|c}
{\mathrm{O}} & {\mathrm{O}} & {\mathrm{O}} & {\mathrm{I}} & {\mathrm{O}}\\
\hline
{\mathrm{I}} & {\mathrm{N}} & {\mathrm{T}}_{q,t+1} & {\mathrm{O}} & {\mathrm{O}}\\
\hline
{\mathrm{O}} & {\mathrm{I}} & {\mathrm{O}} & {\mathrm{O}} & {\mathrm{O}}
\end{array}
\right], {\mathrm{D}}=\left[
\begin{array}{c|c|c}
{\mathrm{I}} & {\mathrm{O}} & {\mathrm{O}}  \\
\hline
{\mathrm{O}} & {\mathrm{I}} & {\mathrm{O}} \\
\hline 
{\mathrm{O}} & {\mathrm{O}} & {\mathrm{I}}  \\
\end{array}
\right]
\end{equation*} 
}}
and applying Lemma~\ref{schur} together with part (g) and (h) of Lemma~\ref{itemize}, we deduce that $\mathrm{M}_o$ is invertible over $\mathbb{F}_2$ if and only if

{\scriptsize{
 \begin{equation*}
 {\mathrm{M}}_o'=\left[
\begin{array}{c|c|c|c|c}
{\mathrm{I}} & {\mathrm{N}} & {\mathrm{T}}_{q,t+1} & {\mathrm{O}} & {\mathrm{I}} \\
\hline
{\mathrm{O}} & {\mathrm{I}} & {\mathrm{O}} &{\mathrm{T}}_{s,1} &  \mathrm{O}  \\
\hline
{\mathrm{T}}_{q,t} & {\mathrm{N}}
& {\mathrm{I}} & {\mathrm{N}} & {\mathrm{I}} \\
\hline
{\mathrm{O}} & {\mathrm{T}}_{s,0} & {\mathrm{O}} & {\mathrm{I}} & {\mathrm{O}}  \\
\hline
{\mathrm{I}} & {\mathrm{O}} & {\mathrm{I}} & {\mathrm{O}} &{\mathrm{O}}  \\
\end{array}
\right]. 
\end{equation*}
}}
is invertible over $\F_2$. We again apply Lemma \ref{schur} on ${\mathrm{M}}_o'$ by taking
{\scriptsize{${\mathrm{D}}=\left[
\begin{array}{c|c|c}
{\mathrm{I}} & {\mathrm{N}} & {\mathrm{I}} \\
\hline
{\mathrm{O}} & {\mathrm{I}} & {\mathrm{O}} \\
\hline
{\mathrm{I}} & {\mathrm{O}} & {\mathrm{O}}
\end{array}
\right]$}} (note that {\scriptsize{${\mathrm{D}}^{-1} \equiv_2 \left[
\begin{array}{c|c|c}
{\mathrm{O}} & {\mathrm{O}} & {\mathrm{I}} \\
\hline
{\mathrm{O}} & {\mathrm{I}} & {\mathrm{O}} \\
\hline
{\mathrm{I}} & {\mathrm{N}} & {\mathrm{I}}
\end{array}
\right]$}} ).  We see that ${\mathrm{M}}_o'$ is invertible over $\F_2$ if and only if $
{\mathrm{M}}^{\prime \prime}_{o}=\left[
\begin{array}{c|c}
{\mathrm{I}} & (t+1){\mathrm{N}} \\
\hline
{\mathrm{O}} & {\mathrm{I}} \\
\end{array}
\right]$ is invertible over $\F_2$. It is easy to see that ${\mathrm{M}}^{\prime \prime}_{o}$ is invertible over $\F_2$. Hence we conclude that ${\mathrm{M}}_o$ is invertible over $\F_2$.
\end{proof}

\begin{proof}[Proof of Theorem \ref{2times1357}]
    Let $n$ be given as in the statement of Theorem \ref{2times1357}. For simplicity in the proof, we regard  $n$ as
$$ n= 2\left(\prod_{i=1}^{t} p_i\right)\left(\prod_{i=1}^{t} q_i\right)\left(\prod_{i=1}^{t} r_i\right)\left(\prod_{i=1}^{t} s_i\right). $$ Then we have 

{\scriptsize{
\begin{equation*}
\label{D21357} 
{\mathrm{D}}_2=\left[
\begin{array}{c|c|c|c}
{\mathrm{O}} & {\mathrm{O}} & {\mathrm{O}} & {\mathrm{O}} \\
\hline
{\mathrm{O}} & {\mathrm{I}} & {\mathrm{O}} & {\mathrm{O}} \\
\hline
{\mathrm{O}} & {\mathrm{O}} & {\mathrm{I}} & {\mathrm{O}} \\
\hline 
{\mathrm{O}} & {\mathrm{O}} & {\mathrm{O}} & {\mathrm{O}} \\
\end{array}
\right],\quad 
{\mathrm{D}}_{-1}=\left[
\begin{array}{c|c|c|c}
{\mathrm{O}} & {\mathrm{O}} & {\mathrm{O}} & {\mathrm{O}} \\
\hline
{\mathrm{O}} & {\mathrm{O}} & {\mathrm{O}} & {\mathrm{O}} \\
\hline
{\mathrm{O}} & {\mathrm{O}} & {\mathrm{I}} & {\mathrm{O}} \\
\hline 
{\mathrm{O}} & {\mathrm{O}} & {\mathrm{O}} & {\mathrm{I}} \\
\end{array}
\right],  \quad 
{\mathrm{E}}=\left[
\begin{array}{c|c|c|c}
 {\mathrm{I}} & {\mathrm{I}} & {\mathrm{O}} & {\mathrm{O}}\\
\hline
{\mathrm{I}} & {\mathrm{I}} & {\mathrm{O}} & {\mathrm{O}} \\
\hline
{\mathrm{O}} & {\mathrm{O}} & {\mathrm{T}}_{r,t+1} & {\mathrm{N}}+{\mathrm{I}} \\
\hline
{\mathrm{O}} & {\mathrm{O}} & {\mathrm{I}} & {\mathrm{T}}_{s,1} \\
\end{array}
\right].    
\end{equation*} 
}}

By definition of Monsky matrix as in \eqref{monskymat}, we have
{\scriptsize{
\begin{equation*}
 {\mathrm{M}}_e\equiv_2\left[
\begin{array}{c|c|c|c|c|c|c|c}
{\mathrm{O}} & {\mathrm{O}} & {\mathrm{O}} & {\mathrm{O}} & {\mathrm{I}} & {\mathrm{I}} & {\mathrm{O}} & {\mathrm{O}}\\
\hline
{\mathrm{O}} & {\mathrm{I}} & {\mathrm{O}} & {\mathrm{O}} & {\mathrm{I}} & {\mathrm{O}} & {\mathrm{O}} & {\mathrm{O}} \\
\hline
{\mathrm{O}} & {\mathrm{O}} & {\mathrm{I}} & {\mathrm{O}} & {\mathrm{O}} & {\mathrm{O}} & {\mathrm{T}}_{r,t} & {\mathrm{N}}+{\mathrm{I}}\\
\hline
{\mathrm{O}} & {\mathrm{O}} & {\mathrm{O}} & {\mathrm{O}} & {\mathrm{O}} & {\mathrm{O}} & {\mathrm{I}} & {\mathrm{T}}_{s,1} \\
\hline
{\mathrm{I}} & {\mathrm{I}} & {\mathrm{O}} & {\mathrm{O}} &{\mathrm{O}} & {\mathrm{O}} & {\mathrm{O}} & {\mathrm{O}} \\
\hline
{\mathrm{I}} & {\mathrm{O}} & {\mathrm{O}} & {\mathrm{O}} & {\mathrm{O}} & {\mathrm{O}} & {\mathrm{O}} & {\mathrm{O}}\\
\hline
{\mathrm{O}} & {\mathrm{O}} & {\mathrm{T}}_{r,t}^{\top} & {\mathrm{I}} & {\mathrm{O}} & {\mathrm{O}} & {\mathrm{I}} & {\mathrm{O}} \\
\hline 
{\mathrm{O}} & {\mathrm{O}} & {\mathrm{N}}+{\mathrm{I}} & {\mathrm{T}}_{s,1}^{\top} & {\mathrm{O}} & {\mathrm{O}} & {\mathrm{O}} & {\mathrm{I}} \\
\end{array}
\right] 
%\xrightarrow[\forall i=1,2,\ldots, t\,]
%{\substack{
%\mathrm{R}_i \leftrightarrow \mathrm{R}_{5+i} \\
%\mathrm{R}_{i+1} \leftrightarrow \mathrm{R}_{4+i}\\}} }\right]. 
\end{equation*}
}}
 Next, we apply the following row operations to $\mathrm{M}_e$. First, interchange the rows $\mathrm{R}_i$ with $\mathrm{R}_{5+i}$ and $\mathrm{R}_{1+i}$ with $\mathrm{R}_{4+i}$ for $1 \leq i \leq t$. Then, for $1 \leq i \leq t$, replace the rows $\mathrm{R}_{5+i}$ by $\mathrm{R}_{4+i}+\mathrm{R}_{5+i}$. Consequently, we obtain
 {\scriptsize{
\begin{equation*}
 {\mathrm{M}}^{\prime}_e = \left[
\begin{array}{c|c|c|c|c|c|c|c}
{\mathrm{I}} & {\mathrm{O}} & {\mathrm{O}} & {\mathrm{O}} & {\mathrm{O}} & {\mathrm{O}} & {\mathrm{O}} & {\mathrm{O}}\\
\hline
{\mathrm{I}} & {\mathrm{I}} & {\mathrm{O}} & {\mathrm{O}} &{\mathrm{O}} & {\mathrm{O}} & {\mathrm{O}} & {\mathrm{O}} \\
\hline
{\mathrm{O}} & {\mathrm{O}} & {\mathrm{I}} & {\mathrm{O}} & {\mathrm{O}} & {\mathrm{O}} & {\mathrm{T}}_{r,t} & {\mathrm{N}}+{\mathrm{I}}\\
\hline
{\mathrm{O}} & {\mathrm{O}} & {\mathrm{O}} & {\mathrm{O}} & {\mathrm{O}} & {\mathrm{O}} & {\mathrm{I}} & {\mathrm{T}}_{s,1} \\
\hline
{\mathrm{O}} & {\mathrm{I}} & {\mathrm{O}} & {\mathrm{O}} & {\mathrm{I}} & {\mathrm{O}} & {\mathrm{O}} & {\mathrm{O}} \\
\hline

{\mathrm{O}} &  \mathrm{I} & {\mathrm{O}} & {\mathrm{O}} & {\mathrm{O}} & {\mathrm{I}} & {\mathrm{O}} & {\mathrm{O}}\\
\hline

{\mathrm{O}} & {\mathrm{O}} & {\mathrm{T}}_{r,t}^{\top} & {\mathrm{I}} & {\mathrm{O}} & {\mathrm{O}} & {\mathrm{I}} & {\mathrm{O}} \\
\hline 
{\mathrm{O}} & {\mathrm{O}} & {\mathrm{N}}+{\mathrm{I}} & {\mathrm{T}}_{s,1}^{\top} & {\mathrm{O}} & {\mathrm{O}} & {\mathrm{O}} & {\mathrm{I}} \\
\end{array}
\right] 
\end{equation*}
}}
Now taking 
{\scriptsize{
\begin{equation*}
{\mathrm{A}}=\left[
\begin{array}{c|c|c|c}
{\mathrm{I}} & {\mathrm{O}} & {\mathrm{O}} & {\mathrm{O}} \\
\hline
{\mathrm{I}} & {\mathrm{I}} & {\mathrm{O}} & {\mathrm{O}} \\
\hline
{\mathrm{O}} & {\mathrm{O}} & {\mathrm{I}} & {\mathrm{O}} \\
\hline 
{\mathrm{O}} & {\mathrm{O}} & {\mathrm{O}} & {\mathrm{O}} \\
\end{array}
\right],{\mathrm{B}}=\left[
\begin{array}{c|c|c|c}
{\mathrm{O}} & {\mathrm{O}} & {\mathrm{O}} & {\mathrm{O}} \\
\hline
{\mathrm{O}} & {\mathrm{O}} & {\mathrm{O}} & {\mathrm{O}} \\
\hline
{\mathrm{O}} & {\mathrm{O}} & {\mathrm{T}}_{r,t} & {\mathrm{N}}+{\mathrm{I}} \\
\hline 
{\mathrm{O}} & {\mathrm{O}} & {\mathrm{I}} & {\mathrm{T}}_{s,1} \\
\end{array}
\right],\quad 
{\mathrm{C}}=\left[
\begin{array}{c|c|c|c}
{\mathrm{O}} & {\mathrm{I}} & {\mathrm{O}} & {\mathrm{O}} \\
\hline
{\mathrm{O}} & {\mathrm{I}} & {\mathrm{O}} & {\mathrm{O}} \\
\hline
{\mathrm{O}} & {\mathrm{O}} & {\mathrm{T}}_{r,t}^{\top} & {\mathrm{I}}  \\
\hline 
{\mathrm{O}} & {\mathrm{O}} & {\mathrm{N}}+{\mathrm{I}} & {\mathrm{T}}_{s,1}^{\top} \\
\end{array}
\right], \quad 
{\mathrm{D}}=\left[
\begin{array}{c|c|c|c}
{\mathrm{I}} & {\mathrm{O}} & {\mathrm{O}} & {\mathrm{O}} \\
\hline
{\mathrm{O}} & {\mathrm{I}} & {\mathrm{O}} & {\mathrm{O}} \\
\hline
{\mathrm{O}} & {\mathrm{O}} & {\mathrm{I}} & {\mathrm{O}}  \\
\hline 
{\mathrm{O}} & {\mathrm{O}} & {\mathrm{O}} & {\mathrm{I}} \\
\end{array}
\right], 
\end{equation*}
}}
Applying Lemma \ref{schur} on $\mathrm{M}^{\prime}_e$, we see that $\mathrm{M}^{\prime}_e$ is invertible over $\F_2$ if and only if 
{\scriptsize{
\begin{equation*}
{\mathrm{M}_{e}^{\prime \prime}}=\left[
\begin{array}{c|c|c|c}
{\mathrm{I}} & {\mathrm{O}} & {\mathrm{O}} & {\mathrm{O}} \\
\hline
{\mathrm{I}} & {\mathrm{I}} & {\mathrm{O}} & {\mathrm{O}} \\
\hline
{\mathrm{O}} & {\mathrm{O}} & {\mathrm{T}}_{r,t}{\mathrm{T}}_{r,t}^{\top}+\mathrm{N}^{2} & {\mathrm{T}}_{r,t}+{
\mathrm{N}}{\mathrm{T}}_{s,1}^{\top}+{\mathrm{T}}_{s,1}^{\top}  \\
\hline 
{\mathrm{O}} & {\mathrm{O}} & {\mathrm{T}}_{r,t}^{\top}+{\mathrm{T}}_{s,1}{\mathrm{N}}+{\mathrm{T}}_{s,1} & {\mathrm{I}}+{\mathrm{T}}_{s,1}{\mathrm{T}}_{s,1}^{\top} \\
\end{array}
\right]
\end{equation*} 
}}
is invertible over $\F_2$.  Hence, by Lemma \ref{itemize}, we have $\mathrm{T}_{r,t}\mathrm{T}_{r,t}^{\top}+ \mathrm{N}^2  \equiv_2 \mathrm{O}$,
and $\mathrm{I}+\mathrm{T}_{s,1}\mathrm{T}_{s,1}^{\top} \equiv_2 \mathrm{N}+\mathrm{I} $. Moreover, by the hypothesis of the theorem, $\mathrm{T}_{r,0}=\mathrm{T}_{s,0}$. Consequently, $\mathrm{T}_{r,t}+\mathrm{N}\mathrm{T}_{s,1}^{\top}+\mathrm{T}_{s,1}^{\top} \equiv_2 t\mathrm{I}$. Since $t$ is odd, applying Lemma \ref{determinant} again, it follows that $\mathrm{M}^{\prime \prime }_{e}$ is invertible over $\mathbb{F}_2$, as required.
\end{proof}

\section{Infinite families and examples}\label{infinite families and examples}
In this section, we study the infinitude of the families appearing in our main results. Define
\begin{equation*}
S_{\ref{157}} (t) := \left\{ n = \prod_{i=1}^{t} p_i q_i r_i \;:\; n \text{ satisfies the conditions of Theorem \ref{157}} \right\}.    
\end{equation*}
Similarly, we define the sets $S_{\ref{355}}$ through $S_{\ref{2times1357}}$ corresponding to Theorems \ref{355} through \ref{2times1357}, respectively. We then have the following result:
\begin{theorem}
For $t \in \mathbb{N}$, the set $S_{\ref{157}} (t)$ is infinite.
\end{theorem}

\begin{proof}
    The proof follows the similar method of \cite[Theorem 2]{ds} which uses the Dirichlet's Theorem of primes in arithmetic progression along with the Chinese remainder Theorem.
\end{proof}
In a similar manner, one can show that there exist infinitely many integers satisfying Theorems \ref{355} through \ref{2times1357}; in other words, each of the sets $S_{\ref{355}}, \ldots, S_{\ref{2times1357}}$ is infinite.

Let $ n \in S_{\ref{157}}(t) $, where the set $ S_{\ref{157}}(t) $ is defined as above. A natural problem is to investigate the asymptotic distribution of the elements of $ S_{\ref{157}}(t) $ within the set of positive integers. More precisely, one may ask for the asymptotic proportion of $ S_{\ref{157}}(t)$, which in turn provides information on the asymptotic proportion of non-congruent numbers satisfying the conditions of Theorem \ref{157}.
%Consider the set $S_\mathcal{F}:= \{n \in \mathbb{N}: n \text{ is square-free}\}$.
%Let $X \subseteq \mathbb{N}$. Then the density of $X$ is defined by $$d(X)= \lim_{x \to \infty}\frac{\# \{m \in X: m \leq x \}}{ \# \{m \in \mathbb{N}: m \leq x\}}$$
Let $k \geq 1$ and $G_k (x) \subseteq \mathbb{N}$ contains all the natural numbers $n \leq x$ such that $n=p_1p_2\cdots p_k$ where $p_i$'s are distinct primes. For two real valued functions $f(x)$ and $g(x)$, we write $f(x) \sim g(x)$ if $\lim\limits_{x \to \infty} \frac{f(x)}{g(x)}=1$. %By \cite[Theorem 6]{hw}, we have $d(X_1) \sim \frac{x}{\log x}$. 
The following result asymptotic proportion of $G_k(x)$.

\begin{theorem}
\cite[Theorem 437]{hw}
\label{6.2}
    Let $k\geq 1$. We have $|G_{k}(x)| \sim \frac{x(\log \log x)^{k-1}}{(k-1)!\log x}$.
\end{theorem}
Using the above result, we have the following: 
\begin{theorem}
    Let $t \in \mathbb{N}$. Define $\bar{S}_{\ref{157}} (t)=\{n \in S_{\ref{157}} (t): n \leq x  \}$. Then 
    $$|\bar{S}_{\ref{157}} (t)|\sim \frac{1}{2^{\frac{9t^2+7t}{2}}}\cdot \frac{x(\log \log x)^{3t-1}}{(3t-1)!\log x}.$$
\end{theorem}

\begin{proof}
    We prove this using induction on $t$. It is easy to see that the induction holds for $t=1$. Indeed, the probability for the prime $p_1$ is $\frac{1}{\phi(8)}$. Since $q_1 \not \equiv b_1^2 \pmod{p_1}$ for any $b_1$ in the residue class of $p_1$, the probability of $q_1$ is $\frac{\phi(p_1)/2}{\phi(8p_1)}=\frac{1}{2 \cdot \phi(8)}$. Similarly, the probability of $r_1$ is also  $\frac{\phi(p_1)/2}{\phi(8p_1)}=\frac{1}{2 \cdot \phi(8)}$. Hence, among integers having three distinct prime factors, the probability of such $n$ is given by $\frac{1}{\phi(8)} \cdot \frac{1}{\left(2\,\phi(8)\right)^2} = \frac{1}{2^8}$.

      Assume that the induction is true for $t=m$. Let $t=m+1$.  We observe that for each prime, apart from their residue class modulo $8$, the Legendre symbol condition contributes $\frac{1}{2}$ in the density factor. Since the choice of the prime $p_{m+1}$ depends on $p_i,q_i, r_i$ for all $1 \leq i \leq m$, the probabily of $p_{m+1}$ is $\frac{1}{2^{3m} \cdot \phi(8)}$. The choice of $q_{m+1}$ depends on $p_i,q_i, r_i$ for all $1 \leq i \leq m$ and $p_{m+1}$ with $q_{m+1} \not \equiv a_{m+1}^2 \pmod{p_{m+1}}$ for any $a_{m+1}$ in the residue class modulo $p_{m+1}$. Hence the probability of $q_{m+1}$ is $\frac{1}{2^{3m+1} \cdot \phi(8)}$. Finally, the choice of $r_{m+1}$ depends on $p_i,q_i, r_i$ for all $1 \leq i \leq m$, $p_{m+1}$ with $r_{m+1} \not \equiv b_{m+1}^2 \pmod{p_{m+1}}$ for any $b_{m+1}$ in the residue class modulo $p_{m+1}$. Hence, using the induction hypothesis, the probability of such $n$ is $$\frac{1}{2^{\frac{9m^2+7m}{2}}} \cdot \frac{1}{2^{3m} \cdot \phi(8)} \cdot \frac{1}{2^{3m+1} \cdot \phi(8)} \cdot \frac{1}{2^{3m+1} \cdot \phi(8)}=\frac{1}{2^{\frac{9(m+1)^2+7(m+1)}{2}}}.$$ Therefore, using Theorem \ref{6.2} we obtain our desired result.
\end{proof}

In a similar way, we can get analogous density results for numbers satisfying Theorem \ref{355} through \ref{2times1357}. We conclude this section by presenting some examples of non-congruent numbers that satisfy the hypotheses of our main results. Note that in Table \ref{tab:examples}, we list all the examples. In the third column, we take  $n = 2^{\epsilon}\prod_{i=1}^{t} p_i q_i r_i s_i,$
where $p_i$, $q_i$, $r_i$, and $s_i$ are primes. For Theorems \ref{157}--\ref{533}, we set $s_i = 1$. Also, for Theorem \ref{377}, we take $p_1 = p$ and $q_1 = q$, and set $p_i = q_i = 1$ for $i > 1$. Moreover, $\epsilon = 1$ for Theorems \ref{533} and \ref{2times1357}, and $\epsilon = 0$ otherwise. Finally, $r(n)$ denotes the algebraic rank of the elliptic curve $E_n: y^2 = x^3 - n^2 x$ over $\mathbb{Q}$. All computations are verified using SageMath \cite{sage}.

\begin{table}
\centering
\caption{Examples corresponding to the main results as in \S \ref{sott}.}
\label{tab:examples}
\begin{tabular}{|l|l|l|l|}
\hline
Theorem & Conditions on $t$, $\alpha$ and $\mu$ & $n=2^{\epsilon}\prod\limits_{i=1}^{t}p_{i}q_{i}r_{i}s_{i}$ & $r(n)$ \\                
\hline

\multirow{2}{*}{Theorem \ref{157}}
& $t=1$ & $(89 \cdot 13 \cdot 23)$ & $0$ \\ \cline{2-4}
& $t=2$, \;\;$\alpha=-1$ & $(17 \cdot 5 \cdot 7) (281 \cdot 389 \cdot 151)$ &  $0$ \\ 
\hline

\multirow{2}{*}{Theorem \ref{355}}
& $t=1$ \;\; 
& $(3 \cdot 5 \cdot 101)$ & $0$\\ \cline{2-4}
& $t=2$,\;\; $\mu=1$
& $(3 \cdot 5 \cdot 29) (59 \cdot 109 \cdot 349)$ & $0$\\ 
\hline

\multirow{2}{*}{Theorem \ref{377}}
& $t=1$,\;\; $\alpha=-1$ 
& $7 \cdot 23 \cdot  3$ & $0$\\ \cline{2-4}
& $t=3$,\;\; $\alpha=\mu=1$ 
& $7 \cdot 31 \cdot (11 \cdot 43 \cdot 347)$ & $0$\\ 
\hline

\multirow{2}{*}{Theorem \ref{533}}
& $t=1$,\;\; $\alpha=\mu=-1$ 
& $2(5 \cdot 3 \cdot 43)$ & $0$\\ \cline{2-4}
& $t=2$,\;\; $\alpha=\mu=1$ 
& $2(5 \cdot 11 \cdot 19) (229 \cdot 491 \cdot 1051)$ & $0$\\ 
\hline

Theorem \ref{1357}
& $t=1$ 
& $(17\cdot 3 \cdot 13 \cdot 47)$ & $0$ \\ 
\hline

Theorem \ref{2times1357}
& $t=1$ 
& $2 (17 \cdot 5 \cdot 19 \cdot 191)$ & $0$ \\ 
\hline

\end{tabular}
\end{table}

%\newpage

\renewcommand{\bibfont}{\footnotesize}
\printbibliography

@article {Li-Qin,
    AUTHOR = {Li, Guilin and Qin, Hourong},
     TITLE = {The {M}onsky matrices and non-congruent numbers},
   JOURNAL = {Algebra Colloq.},
  FJOURNAL = {Algebra Colloquium},
    VOLUME = {31},
      YEAR = {2024},
    NUMBER = {2},
     PAGES = {239--262},
      ISSN = {1005-3867,0219-1733},
   MRCLASS = {11G05 (11D25)},
  MRNUMBER = {4751376},
MRREVIEWER = {Jasbir\ Singh\ Chahal},
       DOI = {10.1142/S1005386724000191},
       URL = {https://doi.org/10.1142/S1005386724000191},
}

@book {Meyer,
    AUTHOR = {Meyer, Carl},
     TITLE = {Matrix analysis and applied linear algebra},
      NOTE = {With 1 CD-ROM (Windows, Macintosh and UNIX) and a solutions
              manual (iv+171 pp.)},
 PUBLISHER = {Society for Industrial and Applied Mathematics (SIAM),
              Philadelphia, PA},
      YEAR = {2000},
     PAGES = {xii+718},
      ISBN = {0-89871-454-0},
   MRCLASS = {15-01},
  MRNUMBER = {1777382},
       DOI = {10.1137/1.9780898719512},
       URL = {https://doi.org/10.1137/1.9780898719512},
}

@book {Silverman,
    AUTHOR = {Silverman, Joseph H.},
     TITLE = {The arithmetic of elliptic curves},
    SERIES = {Graduate Texts in Mathematics},
    VOLUME = {106},
 PUBLISHER = {Springer-Verlag, New York},
      YEAR = {1986},
     PAGES = {xii+400},
      ISBN = {0-387-96203-4},
   MRCLASS = {11G05 (14Gxx 14K07 14K15)},
  MRNUMBER = {817210},
MRREVIEWER = {Robert\ S.\ Rumely},
       DOI = {10.1007/978-1-4757-1920-8},
       URL = {https://doi.org/10.1007/978-1-4757-1920-8},
}

@article {hb,
    AUTHOR = {Heath-Brown, D. R.},
     TITLE = {The size of {S}elmer groups for the congruent number problem.
              {II}},
      NOTE = {With an appendix by P. Monsky},
   JOURNAL = {Invent. Math.},
  FJOURNAL = {Inventiones Mathematicae},
    VOLUME = {118},
      YEAR = {1994},
    NUMBER = {2},
     PAGES = {331--370},
      ISSN = {0020-9910,1432-1297},
   MRCLASS = {11G40 (11G05)},
  MRNUMBER = {1292115},
MRREVIEWER = {Fernando\ Q.\ Gouv\^ea},
       DOI = {10.1007/BF01231536},
       URL = {https://doi.org/10.1007/BF01231536},
}

@article {ds,
    AUTHOR = {Das, Shamik and Saikia, Anupam},
     TITLE = {Families of non-congruent numbers with arbitrarily many pairs
              of prime factors},
   JOURNAL = {Integers},
  FJOURNAL = {Integers. Electronic Journal of Combinatorial Number Theory},
    VOLUME = {20},
      YEAR = {2020},
     PAGES = {Paper No. A55, 12},
      ISSN = {1553-1732},
   MRCLASS = {11G05},
  MRNUMBER = {4130583},
MRREVIEWER = {J.\ Larry\ Lehman},
}

@book {hw,
    AUTHOR = {Hardy, G. H. and Wright, E. M.},
     TITLE = {An introduction to the theory of numbers},
   EDITION = {Fifth},
 PUBLISHER = {The Clarendon Press, Oxford University Press, New York},
      YEAR = {1979},
     PAGES = {xvi+426},
      ISBN = {0-19-853170-2},
   MRCLASS = {10-01},
  MRNUMBER = {568909},
MRREVIEWER = {T.\ M.\ Apostol},
}

@incollection {lagrange,
    AUTHOR = {Lagrange, Jean},
     TITLE = {Nombres congruents et courbes elliptiques},
 BOOKTITLE = {S\'eminaire {D}elange-{P}isot-{P}oitou (16e ann\'ee: 1974/75),
              {T}h\'eorie des nombres, {F}asc. 1},
     PAGES = {Exp. No. 16, 17},
 PUBLISHER = {Secr\'etariat Math., Paris},
      YEAR = {1975},
   MRCLASS = {10B05 (10A40 10B10 14G25)},
  MRNUMBER = {398973},
MRREVIEWER = {Loren\ D.\ Olson},
}

@article {bs,
    AUTHOR = {Birch, B. J. and Stephens, N. M.},
     TITLE = {The parity of the rank of the {M}ordell-{W}eil group},
   JOURNAL = {Topology},
  FJOURNAL = {Topology. An International Journal of Mathematics},
    VOLUME = {5},
      YEAR = {1966},
     PAGES = {295--299},
      ISSN = {0040-9383},
   MRCLASS = {14.40 (10.12)},
  MRNUMBER = {201379},
MRREVIEWER = {J.\ W. S. Cassels},
       DOI = {10.1016/0040-9383(66)90021-8},
       URL = {https://doi.org/10.1016/0040-9383(66)90021-8},
}

@book{genocchi1855note,
  title={Note analitiche sopra tre scritti inediti di Leonardo Pisano pubblicati da Baldassarre Boncompagni},
  author={Genocchi, Angelo},
  year={1855},
  publisher={Tip. delle Belle arti}
}

@article {lindsey,
    AUTHOR = {Reinholz, Lindsey and Yang, Qiduan},
     TITLE = {On the extension of even families of non-congruent numbers},
   JOURNAL = {Rend. Semin. Mat. Univ. Padova},
  FJOURNAL = {Rendiconti del Seminario Matematico della Universit\`a{} di
              Padova},
    VOLUME = {148},
      YEAR = {2022},
     PAGES = {23--49},
      ISSN = {0041-8994,2240-2926},
   MRCLASS = {11G05},
  MRNUMBER = {4542371},
MRREVIEWER = {Mirela\ Juki\'c{} Bokun},
       DOI = {10.4171/rsmup/105},
       URL = {https://doi.org/10.4171/rsmup/105},
}

@article {junguklee,
    AUTHOR = {Kim, Ringi and Lee, Junguk and Lee, Wan and Yu, Myungjun},
     TITLE = {Families of even non-congruent numbers with odd prime factors
              of the form {$8k+5$}},
   JOURNAL = {Tokyo J. Math.},
  FJOURNAL = {Tokyo Journal of Mathematics},
    VOLUME = {47},
      YEAR = {2024},
    NUMBER = {2},
     PAGES = {451--475},
      ISSN = {0387-3870},
   MRCLASS = {11G05 (05C50 15B33)},
  MRNUMBER = {4873812},
MRREVIEWER = {Matteo\ Verzobio},
       DOI = {10.3836/tjm/1502179415},
       URL = {https://doi.org/10.3836/tjm/1502179415},
}

@article {weidong,
    AUTHOR = {Cheng, Weidong and Guo, Xuejun},
     TITLE = {Some new families of non-congruent numbers},
   JOURNAL = {J. Number Theory},
  FJOURNAL = {Journal of Number Theory},
    VOLUME = {196},
      YEAR = {2019},
     PAGES = {291--305},
      ISSN = {0022-314X,1096-1658},
   MRCLASS = {11G05 (11C20 15A03 15B33)},
  MRNUMBER = {3906479},
MRREVIEWER = {Nikola\ Ad\v zaga},
       DOI = {10.1016/j.jnt.2018.09.009},
       URL = {https://doi.org/10.1016/j.jnt.2018.09.009},
}

@article {lindsey2,
    AUTHOR = {Reinholz, Lindsey and Spearman, Blair K. and Yang, Qiduan},
     TITLE = {An extension theorem for generating new families of
              non-congruent numbers},
   JOURNAL = {Funct. Approx. Comment. Math.},
  FJOURNAL = {Uniwersytet im. Adama Mickiewicza w Poznaniu. Wydzia\l\
              Matematyki i Informatyki. Functiones et Approximatio
              Commentarii Mathematici},
    VOLUME = {58},
      YEAR = {2018},
    NUMBER = {1},
     PAGES = {69--77},
      ISSN = {0208-6573,2080-9433},
   MRCLASS = {11G05},
  MRNUMBER = {3780034},
MRREVIEWER = {Andrea\ Ferraguti},
       DOI = {10.7169/facm/1641},
       URL = {https://doi.org/10.7169/facm/1641},
}

@article {ds2,
    AUTHOR = {Das, Shamik and Saikia, Anupam},
     TITLE = {Families of non-congruent numbers with arbitrarily many
              triplets of prime factors},
   JOURNAL = {Kyushu J. Math.},
  FJOURNAL = {Kyushu Journal of Mathematics},
    VOLUME = {78},
      YEAR = {2024},
    NUMBER = {1},
     PAGES = {119--128},
      ISSN = {1340-6116,1883-2032},
   MRCLASS = {11G05 (11A15)},
  MRNUMBER = {4770148},
}

@article {zhang,
    AUTHOR = {Zhang, Shenxing},
     TITLE = {On non-congruent numbers as multiples of non-congruent
              numbers},
   JOURNAL = {J. Algebra},
  FJOURNAL = {Journal of Algebra},
    VOLUME = {687},
      YEAR = {2026},
     PAGES = {394--418},
      ISSN = {0021-8693,1090-266X},
   MRCLASS = {11G05 (11R11 11R29 11R70 14H52)},
  MRNUMBER = {4963309},
MRREVIEWER = {Tomasz\ J\polhk edrzejak},
       DOI = {10.1016/j.jalgebra.2025.09.006},
       URL = {https://doi.org/10.1016/j.jalgebra.2025.09.006},
}

@article{dm-joaa,
    AUTHOR = {Das, Shamik and Mondal, Sudipa},
    TITLE = {A necessary condition for a congruent number of the form $8k+3$},
    JOURNAL = {J. Algebra},
    FJOURNAL = {Journal of Algebra},
    VOLUME = {700},
    YEAR = {2026},
    PAGES = {253--266},
    ISSN = {0021-8693},
    DOI = {10.1016/j.jalgebra.2026.04.009},
    URL = {https://www.sciencedirect.com/science/article/pii/S0021869326001961},
}

@article {Das-Saikia-Rocky,
    AUTHOR = {Das, Shamik and Saikia, Anupam},
     TITLE = {Families of even noncongruent numbers with arbitrarily many
              pairs of prime factors},
   JOURNAL = {Rocky Mountain J. Math.},
  FJOURNAL = {The Rocky Mountain Journal of Mathematics},
    VOLUME = {52},
      YEAR = {2022},
    NUMBER = {2},
     PAGES = {471--481},
      ISSN = {0035-7596,1945-3795},
   MRCLASS = {11G05 (11A15)},
  MRNUMBER = {4422950},
MRREVIEWER = {Jasbir\ Singh\ Chahal},
       DOI = {10.1216/rmj.2022.52.471},
       URL = {https://doi.org/10.1216/rmj.2022.52.471},
}

@manual{sage,
Key = {SAGE},
Author = {W.A. Stein and others},
Organization = {The Sage~Group},
Title = {{S}age {M}athematics {S}oftware ({V}ersion 9.3)},
Year = 2024}

@article {liqin,
    AUTHOR = {Li, Guilin and Qin, Hourong},
     TITLE = {Diophantine equations, class groups and non-congruent numbers},
   JOURNAL = {Ramanujan J.},
  FJOURNAL = {Ramanujan Journal. An International Journal Devoted to the
              Areas of Mathematics Influenced by Ramanujan},
    VOLUME = {62},
      YEAR = {2023},
    NUMBER = {4},
     PAGES = {1081--1105},
      ISSN = {1382-4090,1572-9303},
   MRCLASS = {11G05 (11D25)},
  MRNUMBER = {4667341},
MRREVIEWER = {Andrej\ Dujella},
       DOI = {10.1007/s11139-023-00742-0},
       URL = {https://doi.org/10.1007/s11139-023-00742-0},
}

\end{document}